\documentclass[runningheads, envcountsame,a4paper]{article}

\usepackage{amssymb}
\usepackage{graphicx}
\usepackage{caption}
\usepackage{subcaption}
\usepackage{enumerate}
\usepackage{amsmath,amsopn,amsfonts}
\usepackage[hyphens]{url}
\usepackage{microtype}
\usepackage{amsthm}
\usepackage{tikz}
\usetikzlibrary{patterns}
\usepackage{comment}
\usepackage{hyperref}
\usepackage{tikz-cd}
\hypersetup{
	colorlinks=true,
	linkcolor=blue,
	filecolor=magenta,      
	urlcolor=cyan,
}

\usepackage{url}

\usepackage{url}
\urldef{\mailsa}\path| berdinsky@gmail.com, 
prohrakju@gmail.com |

\date{}

\newtheorem{definition}{Definition}
\newtheorem{example}[definition]{Example}
\newtheorem{remark}[definition]{Remark} 
\newtheorem{theorem}[definition]{Theorem}

\newtheorem{lemma}[definition]{Lemma}
\definecolor{light-gray}{gray}{0.75}

\begin{document}

\title{Extending the Synchronous Fellow 
Traveler Property}

\author{Prohrak Kruengthomya
	\thanks{Mahidol University, Faculty of Science, 
		Department of Mathematics,  
		Bangkok 10400, Thailand; e--mail: 
		prohrakju@gmail.com.} \and
	Dmitry  Berdinsky
	\thanks{Mahidol University, Faculty of Science,
		Department of Mathematics and 
		Centre of Excellence in Mathematics, CHE, 
		Bangkok 10400, Thailand;
		e--mail: berdinsky@gmail.com.} }
\date{}
\maketitle {\small
	\begin{quote}
		\noindent{\bf Abstract}\,\,\,\,   
	  	
	  	We introduce an
	  	extension of the fellow
	  	traveler property 
	  	which allows fellow 
	  	travelers to be at 
	  	distance bounded from above by a function $f(n)$ 	  	
	  	growing slower than 
	  	any linear function.
	    We study normal forms 
	    satisfying this 
	    extended fellow 
	    traveler property 
	    and certain geometric constraints that 
	    naturally generalize 
	    two fundamental properties of an
	    automatic normal form -- the regularity of its language and 
	    the bounded length 
	    difference property.      
	    We show examples of such normal forms
	    and prove some non--existence theorems. 
	    
	    

	\end{quote}

\section{Introduction}	

The fellow traveler property is a cornerstone 
of the theory of automatic groups 
introduced by Thurston
and others \cite{Epsteinbook}.
A normal form of a group
satisfies the fellow traveler 
property if  every two normal forms 
of group elements which are at distance one 
(with respect to some fixed set of 
generators) are $k$--fellow travelers for 
some positive integer $k$. 
The latter, informally speaking,
means that when  
two fellow travelers are synchronously moving 
with the same speed
along the paths labeled by such normal forms 
they are always at distance at most 
$k$ from each other.    
Each automatic group admits   
an automatic structure that includes 
a normal form which satisfies
the fellow traveler property 
\cite[\S~2.3]{Epsteinbook}.
If a normal form of a group satisfies 
the fellow traveler property and 
its language is regular, then the  group must be automatic. 

The idea of extending automatic 
groups while retaining the fellow   
traveler property or its natural relaxations 
is not new.
Requiring the fellow traveler property for 
a normal form
but dropping the regularity condition 
for its language leads to the notions of a combing and a combable
group. Bridson showed that 
there exist combable
groups which are not automatic
\cite{Bridson03}\footnote{Note that the  notion of a combable group used by Bridson  is different from the one originally 
	introduced by Epstein and Thurston  \cite[\S~3.6]{Epsteinbook}.}. 
A more general approach is to allow 
fellow travelers moving with different speeds. 
This leads to the notion of the asynchronous 
fellow traveler property. Thurston showed  
that the Baumslag--Solitar group  
$BS(p,q)$ for $1 \leqslant p < q$ admits 
a normal form which satisfies the asynchronous
fellow traveler 
and the language of this normal form is regular \cite[\S~7.4]{Epsteinbook}. 
Bridson and Gilman 
showed that the fundamental 
group of a compact $3$--manifold 
admits a normal form  
which satisfies
the asynchronous fellow traveler property 
and its language is indexed \cite{BridsonGilman1996}.

In this paper we consider a 
relaxation of the fellow traveler 
property requiring fellow 
travelers to move with the same 
speed (synchronously) but 
allowing them to be at distance bounded from 
above by $f(n)$, where  
$f : \mathbb{N} 
     \rightarrow 
     \mathbb{R}_+$ 
is a function growing slower than any linear function, e.g., 
$n^{\alpha}$ for $0 < \alpha < 1$ 
or $\log(n)$,
and $n$ is the distance that 
fellow travelers traversed 
starting from the origin, see 
Definition \ref{fellow_traveler}. 
In this case we  
say that a normal form 
satisfies the 
$f(n)$--fellow traveler 
property.  
First in the 
subsection \ref{two_ways_fellow_prop}
we notice that such normal
forms can be trivially constructed.    
So then in the subsection \ref{quasigeo_quasireg_subsec} 
we introduce two types of geometric constraints      
each of which makes
constructing such normal forms nontrivial. 
The first  
geometric constraint  requires a normal 
form to be quasigeodesic. 
It is a
relaxation of the bounded 
length difference 
property\footnote{Note the 
bounded length difference 
property is referred to as 
the comparable length property 
in \cite{Bridson03}. Also, 
the bounded length difference 
property is incorporated in the notion of a combable group introduced by Epstein and Thurston. 
It is an open question whether there exists a combable group
in the sense of Epstein 
and Thurston which is not
automatic.}
\cite[Lemma~2.3.9]{Epsteinbook}
which requires the 
geodesic length of a group 
element and the length of its 
normal form to be linearly 
comparable, see 
Definition \ref{quasi_def}. 
The second geometric 
constraint requires 
a normal form to be quasiregular. It is a relaxation of the 
regularity condition for 
the language of a normal form
which requires that for each 
prefix $u$ 
of a normal form of
a group element there exists 
a word $v$ of length at 
most $c>0$ for which 
$uv$ is a normal form of
some group element, see 
Definition \ref{quasireg_def}. 
Alternatively,  
quasiregularity  
can be considered as a 
relaxation of 
prefix--closedness;  
see also Definition \ref{quasiprefclos_def} and Theorem \ref{quasirefequalquasiprefclos}.

The main results of the paper 
are as follows. 
In Theorems 
\ref{ssp_no_quasidesic_thm}
and
\ref{nonfp_no_quasigeodesic_thm}
we show that there exists 
no quasigeodesic normal form
which satisfies the  
$f(n)$--fellow traveler 
property in 
a finitely presented 
group with the
strongly--super--polynomial 
Dehn function and 
a non--finitely 
presented group, respectively;
see Definition \ref{ssp_def}
for the notion of a 
strongly--super--polynomial 
function. 
In Theorem \ref{relation_cayley_automatic_thm} we show relation between 
the notion of a 
Cayley distance function 
studied by Elder, Taback, 
Trakuldit and the second author
\cite{BET22,eastwest19,LATA18} 
and quasigeodesic normal 
forms satisfying 
the $f(n)$--fellow traveler property. Namely,
Theorem
\ref{relation_cayley_automatic_thm}  shows that if 
a non--automatic group 
has a Cayley 
automatic representation 
for which the Cayley distance function $h(n)$ grows 
slower than a linear 
function\footnote{The existence of such Cayley automatic representations is an open question.}, 
then this group admits   
a quasigeodesic normal form 
satisfying the $h(n)$--fellow
traveler property. 
Theorem \ref{BS_qusireg_normal_form}
shows that for $1 \leqslant p < q$
the Baumslag--Solitar group 
$BS(p,q) = \langle a, t \, 
| \, t a^p t^{-1} = a^q \rangle$
admits a prefix--closed 
(so quasiregular) normal
form satisfying the
$\log (n)$--fellow traveler 
property. 
Theorem \ref{Z2wrZ2_quasireg_norm_form} 
shows that the wreath 
product 
$\mathbb{Z}_2 \wr \mathbb{Z}^2$ 
admits a prefix--closed 
normal form satisfying the
$\sqrt{n}$--fellow traveler 
property\footnote{Note that
by Theorems
\ref{ssp_no_quasidesic_thm}
and
\ref{nonfp_no_quasigeodesic_thm}
the groups $BS(p,q)$ for 
$1 \leqslant p < q$ and 
$\mathbb{Z}_2 \wr \mathbb{Z}^2$
do not admit quasigeodesic 
normal forms satisfying the
$f(n)$--fellow traveler 
property.}. 

The rest of the paper is organized 
as follows. 
In Section \ref{preliminaries_sec}
we recall the notions of
a group, a normal form,  
an automatic group, the fellow traveler property and 
the relations 
$\preceq$ and $\ll$
for nondecreasing functions
appeared in 
this paper.  
In Section \ref{fellow_sec} 
we introduce the
$f(n)$--fellow traveler 
property, quasigeodesic  
and quasiregular normal forms. 
In Section \ref{quasigeodesic_sec} 
for quasigeodesic
normal forms satisfying 
the $f(n)$--fellow traveler 
property
we prove the  non--existence 
theorems  in groups with the 
strongly--super--polynomial
Dehn function and 
non--finitely presented 
groups and show relation with 
the notion of a Cayley distance
function. 
In Section \ref{quasiregular_sec} 
we show examples of 
quasiregular
normal forms satisfying 
the $f(n)$--fellow traveler 
property.  
Section \ref{conclusion_sec} concludes the paper.  

\section{Preliminaries}
\label{preliminaries_sec} 
 
 In this section we introduce necessary 
 notations and  recall some definitions.

 {\it Groups and normal forms}. Let $G$ be a finitely generated infinite group and 
 $A = \{a_1, \dots, a_m \}$, where 
 $a_i \in G$ for $i = 1, \dots, m$,  
 be a finite set generating $G$: 
 each  element of $G$ can 
 be written as a product 
 of elements from $A$ and 
 their inverses. We allow different 
 $a_i, a_j \in A$, $i \neq j$, 
 to be equal in $G$ and  
 some elements in $A$ to be 
 equal the identity $e \in G$. 
 We denote by $A^{-1}$ the set of formal inverses 
 for elements in $A$: 
 $A^{-1}= \{a_1 ^{-1}, \dots, a_m ^{-1}\}$ 
 and  by $S$ the union of $A$ 
 and  $A^{-1}$: $S = A \cup A^{-1}$. 
 We denote by  
 $d_A$ the word metric in $G$ relative 
 to $A$: for $g_1, g_2 \in G$ 
 the distance $d_A (g_1, g_2)$ is the 
 length of a shortest word 
 $u \in S^*$  equal to 
 $g_1 ^{-1} g_2$ in $G$.  
 For a given $g \in G$ we denote by 
 $d_A (g)$ the distance between $g$ and 
 the identity $e \in G$  
 with respect to $d_A$:  
 $d_A(g) = d_A (e,g)$. 
 For a given $w  = s_1 \dots s_\ell \in S^*$, 
 we denote by 
 $|w|$ the length of $w$: 
 $|w|=\ell$ and by $\pi(w)$ 
 the group element  $s_1 \dots s_\ell \in G$, 
 where $\pi$ refers to the canonical 
 projection map $\pi : S^* \rightarrow G$.  
 
 A normal form of $G$ is a rule for assigning 
 a word $w \in S^*$ to a group element 
 $g \in G$ such that $\pi (w) = g$. The word $w$ is referred to as 
 a normal form of $g$. 
 In this paper we always assume 
 that a normal form is one--to--one:
 for each $g \in G$ exactly one word  
 $w \in S^*$ is assigned. 
 A normal form defines  
 a language $L \subseteq S^*$.  
 Similarly, a language $L \subseteq S^*$
 for which the restriction  
 $\pi_L : L \rightarrow G$  is
 surjective and one--to--one defines a 
 normal form of $G$.   
  
 \vskip1mm
 
 {\it The fellow traveler property and automatic groups.}    
  Let $w \in S^*$ be a word and 
  $t \in [0,+\infty)$ be a nonnegative integer. 
  We define 
  $w(t)$ to be the prefix of $w$ of a
  length $t$ if $t\leqslant |w|$ and 
  $w$ if $t>|w|$.  
  Let $L \subseteq S^*$ be a normal 
  form of $G$.    
  We denote by $s(n)$ 
  a function 
  $s : [0, + \infty) \rightarrow \mathbb{R}_+$ 
  defined as:
  \begin{equation}
  	\label{def_function_f_1}  
  	s(n)= \max \{d_A (\pi(w_1(t)),
  	\pi (w_2(t))) \, | \, t \leqslant n, \, w_1, w_2 \in L, 
  	a \in A , \pi(w_1)a = \pi(w_2)\}.
  \end{equation} 
  The function $s(n)$ is the maximum distance 
  between fellow travelers moving 
  with the same speed along 
  the paths in $\Gamma (G,A)$ 
  labeled by words $w_1 $ and 
  $w_2$ for which 
  $\pi (w_1) a = \pi (w_2)$ for some $a \in A$. 
  \begin{definition}[the fellow traveler property] 
  \label{standard_fellow_def}	 
  	 It is said that the normal form 
  	 $L \subseteq S^*$ 
  	 satisfies the fellow traveler 
  	 property if  the function
  	 $s(n)$ given by \eqref{def_function_f_1} 
  	 is bounded from above by a constant.  
  \end{definition}	
\noindent Recall that the group $G$ is called automatic if 
it admits a normal form 
$L \subseteq S^*$ for which the language 
$L$ is regular and for each 
$a \in A$ the binary relation 
$R_a = \{ (u,v) \,|\, \pi(u)a = \pi(v)\}$ 
is recognized by a two--tape synchronous 
automaton \cite{Epsteinbook}. 
In this case the normal form $L$ satisfies 
the fellow traveler property. 
Equivalently, if $G$ admits a normal form $L \subseteq S^*$  which satisfies the fellow traveler property
and $L$ is regular, then $G$ must be automatic 
\cite[Theorem~2.3.5]{Epsteinbook}. 
In this paper we focus on groups which 
are not automatic (non--automatic groups).     
  
\vskip1mm

 {\it The relations $\preceq$ and $\ll$ for 
 nondecreasing functions}.	
We denote by $\mathbb{N}$ a set of 
all natural numbers which includes zero. 
For a given $N \in \mathbb{N}$ 
we denote by $[N,+ \infty)$  
the set 
$\left[N,+\infty\right) = 
\{n \in \mathbb{N}\,|\, n \geqslant N\}$.  
Let $\mathcal{F}$ be the set of 
all nondecreasing functions 
$ f: [N, + \infty ) \rightarrow \mathbb{R}_{+}$, 
where  $\mathbb{R}_{+} = \{x \geqslant 0\,|\, 
x \in \mathbb{R} \}$.  

\begin{definition}[$\preceq$ relation]
	\label{coarseineqdef}   
	For given $f,h \in \mathcal{F}$     
	we say that $h \preceq f$ if 
	there exist positive integers  
	$K,M$ and a nonnegative $N$    
	such that     
	$h (n) \leqslant K f(M n)$
	for all  $n \in \left[N, + \infty \right)$.    
	We say that 
	$h \asymp f$ if $h \preceq f$ and 
	$f \preceq h$. 
	If $h \preceq f$ and $h \not\asymp f$, 
	we say that $h \prec f$. 
\end{definition}

\begin{definition}[$\ll$ relation] 
	For given $f,h \in \mathcal{F}$ we say that 
	$h \ll f$ if there exists an unbounded  
	function $t \in \mathcal{F}$ such that 
	$ht \preceq f$. 
\end{definition}	 

\noindent We denote by 
$\mathfrak{i} : [0, + \infty ) \rightarrow 
\mathbb{R}_+$ the identity function: 
$\mathfrak{i}(n)=n$.
Note that 
$f \ll \mathfrak{i}$ is stronger 
than $f \prec \mathfrak{i}$, that is,    
$f \ll \mathfrak{i}$ implies  
$f \prec \mathfrak{i}$. Indeed, $f \ll \mathfrak{i}$ 
implies $f \preceq \mathfrak{i}$. 
Now suppose that $f \ll \mathfrak{i}$ and 
$\mathfrak{i} \preceq f$. 
The inequality $\mathfrak{i} \preceq f$ 
implies that there exist positive integers 
$K,M$ and a nonnegative integer $N$ 
such that $n \leqslant K f(Mn)$ for 
all $n \in \left[ N ,+ \infty \right)$.     
The inequality $f \ll \mathfrak{i}$ implies that there exist  positive integers $K',M'$ and a nonnegative
integer $N'$ such that  
$f(n) t(n)\leqslant K'M' n $ for all 
$n \in \left[N',+\infty \right)$, where $t(n)$ 
is some unbounded function. 
Therefore, $f(n) \leqslant \frac{K'M'}{t(n)}n$ 
for all $n \in \left[ N', +\infty\right)$, which implies that $Kf(Mn) \leqslant \frac{KK'MM'}{t(Mn)}n$ also 
for all $n \in \left[ N', +\infty\right)$. 
Since $t(n)$ is unbounded,  
we get a contradiction. Therefore, 
we have that
$\mathfrak{i} \not\preceq f$, so 
$\mathfrak{i} \prec f$.   
The reverse ($f \prec \mathfrak{i}$ 
implies $f \ll \mathfrak{i}$) 
in general is not true 
as it is shown in Example \ref{example_ll_vs_prec}.   
\begin{example}
	\label{example_ll_vs_prec}	  
	Let $n_i, i \geqslant 1$ be an infinite 
	sequence defined recursively by the 
	identities: 
	$n_0=0$, $n_1 =1$ and $n_{i+1}= n_i 2^{2i}$ 
	for $i \geqslant 1$.  
	Let $f(n)$ be a function for which 
	$f (n) = 2 n_i$ for $n_i \leqslant n < n_{i+1}$.      
	Clearly, $f(n) \preceq \mathfrak{i}$. 
	Let us show that 
	$\mathfrak{i} \not \preceq f(n)$. 
	The inequality $\mathfrak{i} \preceq f$ 
	implies that  
	there exist positive integers 
	$K,M$ and $N$ such that 
	$n \leqslant K f(Mn)$ for all 
	$n \in \left[N, + \infty \right)$. 
	Therefore, if $n_i 2^i \geqslant N$, 
	then $n_i 2^i \leqslant K f(M n_i 2^i)$. 
	By the definition of $f(n)$, if $M < 2^i$, then $f(M n_i 2^i) = 2n_i$.   
	Therefore, if $n_i 2^i \geqslant N$ 
	and $M < 2^i$, then $n_i 2^i \leqslant 2Kn_i$.  
	The latter inequality is true only if $2^i \leqslant 2K$ and false, otherwise. 
	Thus, $\mathfrak{i} \not\preceq f$.  
	Therefore, $f \prec \mathfrak{i}$.  	
	Let us show now 
	that $f \not\ll \mathfrak{i}$.
	The inequality $f \ll \mathfrak{i}$ implies 
	that there exist positive integers 
	$K,M$ and a nonnegative integer 
	$N$ such that $t(n)f(n) \leqslant 
	KM n$ for some unbounded $t \in \mathcal{F}$ 
	and all $n \in \left[N, +\infty\right)$.
	Since $t(n)$ is unbounded, there exists 
	$i_0$ such that for all $i \geqslant i_0$, 
	$t(n_i) \geqslant KM$ which implies 
	that $f(n_i) = 2n_i \leqslant n_i$ for 
	all $i \geqslant i_0$. The latter is 
	impossible for $i >0$. 
\end{example}	


\section{The $f(n)$--Fellow 
Traveler Property}
\label{fellow_sec}

 We extend the fellow traveler 
 property by allowing the distance 
 between fellow travelers 
 to be bounded from above by a 
 nondecreasing function 
 $f(n)$.      
 If $f(n)$ is bounded from above
 by a constant, 
 then  one simply gets  
 the fellow traveler property, see Definition
 \ref{standard_fellow_def}.  
 We will allow the function $f(n)$ 
 to be unbounded.
 
\begin{definition}[the $f(n)$--fellow traveler property]
\label{fellow_traveler}	
	Let $f \in  \mathcal{F}$ be a function  
	for which $f \ll \mathfrak{i}$.   
    We say that a 
    given normal form $L \subseteq S^*$ 
    of a group $G$ 
    satisfies the 
    $f(n)$--fellow traveller property 
    if for the function $s(n)$ given 
    by \eqref{def_function_f_1} 
    the  inequality   
    $s \preceq f$
    holds.             
\end{definition}

 \noindent In Definition 
 \ref{fellow_traveler} the function 
 $f(n)$ is an upper bound for the 
 distance between fellow travelers 
 in coarse sense. It can be verified 
 that Definition \ref{fellow_traveler} 
 does not depend on the choice 
 of the set of generators $A$.   
 By the triangle inequality, 
 for every normal form in $G$
 we always have 
 that
 $s(n) \leqslant 2n$ for all 
 $n \in \mathbb{N}$. 
 From Example \ref{example_ll_vs_prec} we see that there exists a function $f \prec \mathfrak{i}$ for which 
 $f(n)= 2n$ at infinitely many
 points. If 
$f \ll \mathfrak{i}$, then, informally speaking, $f$  genuinely 
grows slower than $\mathfrak{i}$. 
This explains the choice of the inequality 
$f \ll \mathfrak{i}$ over    
the inequality $f \prec \mathfrak{i}$ in 
Definition \ref{fellow_traveler}.
For illustration of the $f(n)$--fellow traveler 
property see Fig.~\ref{fellow_pic}.  

 \begin{figure}
	\centering 	
	\includegraphics[width=6.5cm]{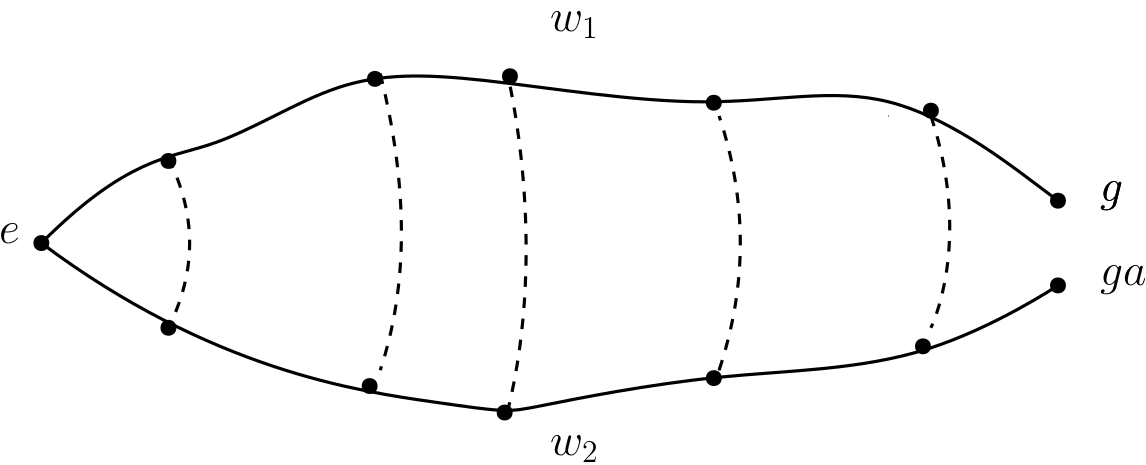}
	\caption{The upper and lower curves show 
	         the paths labeled by the 
	         normal forms $w_1$ and $w_2$
	         of group elements $g$ and $ga$, 
	         respectively. 
	         The pairs of dots and dashed curves connecting  it  
	         show fellow travelers and
	         shortest paths between them, 
	         respectively.}
	\label{fellow_pic}
\end{figure}


\subsection{Normal forms satisfying the 
	$f(n)$--fellow traveler property}
\label{two_ways_fellow_prop}

 We now show two  ways 
 of modifying a given normal form 
 so the modified normal form 
 satisfies the $f(n)$--fellow traveler 
 property for some 
 $f \ll \mathfrak{i}$. 
 Let us be given a normal form of 
 a group $G$ defined by a language 
 $L \subseteq S^*$. 
 

 
 \vskip1mm
 
 {\it First way}.
 Let $u \in S^*$ be a word defining a loop in $\Gamma(G,S)$: $\pi (u) = e$. 
 We define a language $L' \subseteq S^*$ as 
 $L' = \{ u^{\ell^2} w\,|\, w \in L \wedge 
          \ell = |w| \}$. 
 That is, for every $w \in L$ we attach a prefix 
 $u^{\ell^2}$ to the word $w$, 
 where $\ell = |w|$. See  Figure \ref{way1} for 
 illustration.  
 For a normal form given by a language $L'$ the 
 function $s(n)$ is bounded from above 
 by $\sqrt{n}$: $s(n) \preceq \sqrt{n}$. 
 Indeed, let us consider two words 
 $w_1 \in L$ and $w_2 \in L$ for which 
 $\pi (w_1)a = \pi (w_2)$. 
 Let $|u|=c$, $\ell_{1}=|w_{1}|$ and $\ell_{2}=|w_{2}|$. 
 Without loss of generality we assume that $\ell_1 \leqslant\ell_2$. 
 We denote by $n$ a number of steps traversed by fellow travelers 
 along the paths labeled by $w_1$ and $w_2$. 
 There are  three different cases to consider: 
 \begin{itemize}
 \item{If $n\leqslant c\ell_1^{2}$, then the distance between fellow travelers is bounded from above by $2c$.}
 \item{If $c\ell_1^{2}<n\leqslant c\ell_2^{2}$, 
 then the distance between fellow travelers is bounded from above by $(\ell_{1}+c)$, so it is 
 strictly less than $\sqrt{n/c}+c$.}
 \item{If $c\ell_2^{2}<n\leqslant c\ell_2^{2}+\ell_2$, then the distance between fellow travelers is bounded from above by $\ell_1$+$\ell_2 \leqslant 2\ell_2$, so it is strictly less than $2 \sqrt{n/c}$.} 
 \end{itemize}
 From these three cases 
 it can be seen that $s(n) \preceq \sqrt{n}$ 
 for the normal form defined by the language $L'$.
 
  \begin{figure}
 	\centering 	
 	\includegraphics[width=6.5cm]{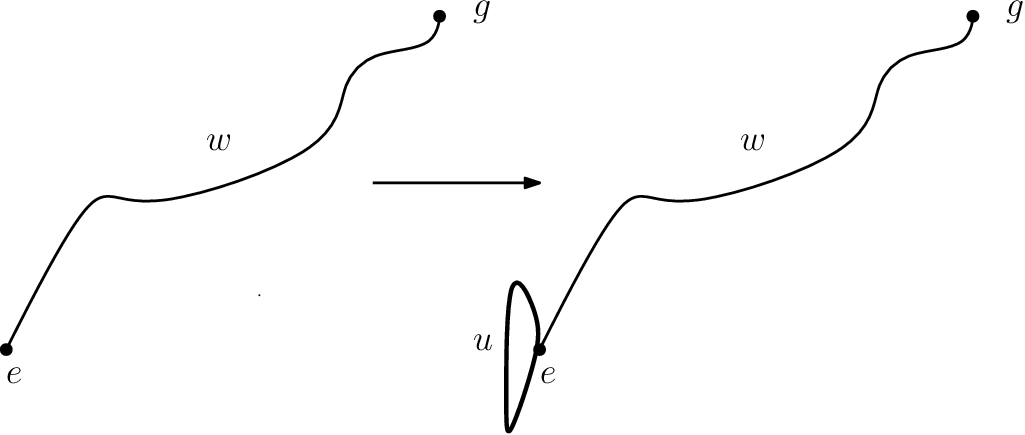}
 	\caption{A curve on the left shows 
 	the path labeled by a normal 
 	form $w$ of a group element $g$. A curve 
 	on the right shows the path labeled by
 	the modified normal 
 	of $g$: first it traverses 
 	the loop labeled by $u$ doing it  
 	$\ell^2$ times and then it traverses 
 	the path labeled by $w$.}
 	\label{way1}
 \end{figure}
 
 \vskip1mm 
 
 {\it Second way.}    
 Let $w = s_1 \dots s_m$ be a nonempty word 
 in the language $L$. 
 For each integer $k \in \{1, \dots, m \}$, 
 let us choose a loop $u_k \in S^*$ 
 for which 
 $ C_1 \sqrt{k} \leqslant |u_k| \leqslant C_2 \sqrt{k}$ for some 
 fixed constants $C_1,C_2 >0$. 
 These loops $u_k$, 
 $k = 1,\dots,m$ can be chosen arbitrarily for each $w \in L$ and 
 $k \in \{1,\dots,m\}$, 
 where $m = |w|$. 
 Let $w' =s_1 u_1 s_2 u_2 
 \dots s_m u_m$. If $w= \varepsilon$, 
 then we put $w' = \varepsilon$. 
 We define a language 
 $L'' \subseteq S^*$ as 
 $L'' = \{ w' \,|\, w \in L\}$.           
 See Figure \ref{way2} for illustration. 
 
 For a normal form given by a language $L''$ 
 the function $s(n)$ is bounded from above 
 by $\sqrt{n}$: $s(n) \preceq \sqrt{n}$.  
 Indeed, let us consider two words 
 $w_1 = s_1 \dots s_{m_1} \in L$ and 
 $w_2 = t_1 \dots t_{m_2} \in L$
 for which $\pi(w_1) a = \pi(w_2)$.
 For the words $w_1$ and $w_2$, 
 let $w_1 ' = s_1 u_1 \dots s_{m_1} u_{m_1}$
 and $w_2 ' = t_1 v_1 \dots t_{m_2} v_{m_2}$, 
 respectively.  
 Let $n$ be a number of steps 
 traversed  by fellow travelers 
 along the paths labeled by $w_1$ 
 and $w_2$. 
 We denote by  
 $k_1$ the   
 integer for which 
 $w_1 (n) = s_1 u_1 \dots 
            s_{k_1} u_{k_1}  q_1$, 
 where either $k_1= m_1$ and $q_1 = \varepsilon$ 
 or $q_1$ is a proper prefix of 
 $s_{k_1+1}u_{k_1+1}$. 
 Similarly, we denote by $k_2$ the integer 
 for which 
 $w_2 (n)= t_1 v_1 \dots t_{k_2} v_{k_2} q_2$, 
 where either $k_2=m_2$ and 
 $q_2 = \varepsilon$ or $q_2$ is a proper 
 prefix of $t_{k_2+1}v_{k_2+1}$.   
 Now the distance between fellow travelers 
 is bounded from above by 
 $k_1 + |q_1| + k_2 + |q_2| \leqslant
  k_1 + (1 + C_2 \sqrt{k_1 +1}) + k_2 + 
  (1 + C_2 \sqrt{k_2 +1})$. On the other hand, 
  we have that 
  $C k_1 ^{\frac{3}{2}} \leqslant n$ and 
  $C k_2 ^{\frac{3}{2}} \leqslant n$ for 
  some constant $C>0$. Therefore, 
  $s(n) \preceq n^{\frac{2}{3}}$ for the 
  normal form defined by the language $L''$.
   
  \begin{figure}
  	\centering 	
  	\includegraphics[width=6.5cm]{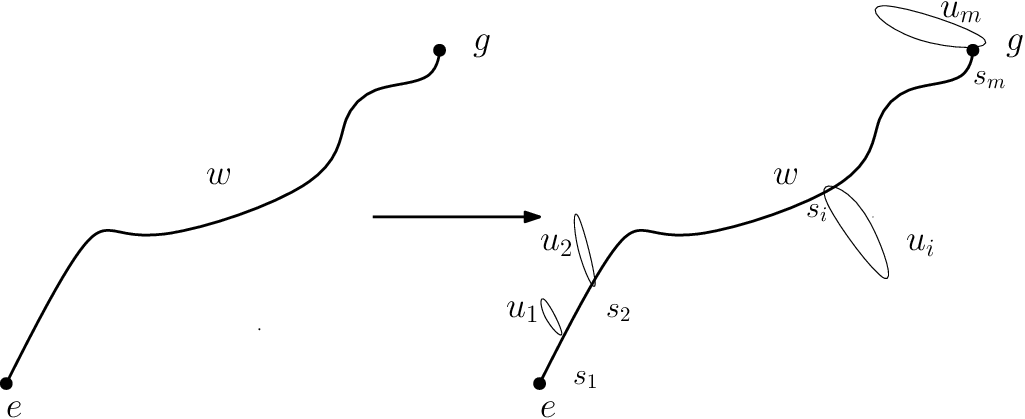}
  	\caption{A curve on the left shows 
  		the path labeled by a 
  	normal form 
  	$w =s_1 s_2 \dots s_m $
  	of a group element $g$. 
  	A curve on the right shows 
  	the path labeled by 
  	the	modified normal form of $g$: 
  	$s_1 u_1 s_2 u_2 \dots s_i u_i 
  	 \dots s_m u_m$.}
  	\label{way2}
  \end{figure}

\subsection{Quasigeodesic and 
quasiregular normal forms} 
\label{quasigeo_quasireg_subsec}

We introduce two kinds of geometric constraints for a normal form  
which originate 
from two basic properties 
of an automatic normal form:
the bounded length difference
property
\cite[Lemma~2.3.9]{Epsteinbook} and the regularity of its language. 
These geometric constraints
in general are an obstacle for trivial constructions of  normal forms satisfying 
the $f(n)$--fellow
traveler property shown in the 
subsection \ref{two_ways_fellow_prop}. 
In this paper normal forms  
satisfying these geometric 
constraints are referred to as 
quasigeodesic and quasiregular.  

\vskip1mm 

{\it Quasigeodesic normal forms}. Recall that 
the bounded length difference property for the 
normal form defined by a language 
$L \subseteq S^*$ 
means that there exists a constant 
$C'>0$ such that for every 
$w_1,w_2 \in L$ for which 
$\pi(w_1) a = \pi (w_2)$ for 
some $a \in A$ the inequality 
$||w_1| - |w_2|| \leqslant C'$ holds.
Let $C= \max  \{C', |w_0| \}$, where 
$w_0$ is the normal form of 
the identity $e \in G$. Then the bounded 
length difference property implies that 
for every $w \in L $ we have 
$|w| \leqslant C (d_A (\pi (w))+1)$.
The latter is the notion of a quasigeodesic normal form
introduced by Elder and Taback \cite{ElderTabackCgraph},
see Definition \ref{quasi_def}.
Note that 
for a normal form being quasigeodesic, in general, 
does not imply having the bounded 
length difference property.   
\begin{definition}[quasigeodesic normal form] 
\label{quasi_def}   
   A normal form defined by a language 
   $L \subseteq S^*$ 
   is said to be quasigeodesic 
   if there is a constant 
   $C>0$  such that for every $w \in L$ 
   the inequality 
   $|w| \leqslant C (d_A (\pi (w)) + 1)$
   holds.         
\end{definition}
\noindent Note that if $L \subseteq S^*$ is a quasigeodesic normal form in the sense of Definition \ref{quasi_def}, 
the paths labeled by elements of $L$ might not be 
quasigeodesics in the standard sense, i.e., not every 
subword $u'$ of $u \in L$ needs
to satisfy the inequality $|u'| \leqslant C (d_A (\pi(u'))+1)$.
We notice that normal forms 
constructed in the subsection \ref{two_ways_fellow_prop} are not quasigeodesic. Indeed, for the first way 
we have that the normal form of a group 
element $\pi (w)$ is $w' = u^{\ell^2} w$, where 
$\ell = |w|$. Since 
$d_A (\pi (w)) \leqslant \ell$ and 
$|w'| = |u|\ell^2 + \ell \geqslant \ell^2$, 
the inequality 
$|w'| \leqslant C(d_A (\pi(w))+1)$ cannot 
hold for all $w \in L$ and some constant $C>0$. 
For the second way we have that the normal 
form of a group element $\pi (w)$ is 
$w' = s_1 u_1 s_2 u_2 \dots s_m u_m$.
Since $d_A(\pi(w)) \leqslant m$ and
$w' = m + |u_1| + \dots + |u_m| \geqslant 
C_1 \frac{m}{2} \sqrt{\frac{m}{2}}$, 
the inequality 
$|w'| \leqslant C(d_A (\pi (w))+1)$ cannot
not hold for all 
$w \in L$ and some constant $C>0$.  


 \vskip1mm 
 
 {\it Quasiregular normal forms}. 
 In Definition  \ref{quasireg_def} 
 we introduce the notion of a 
 quasiregular  normal form.   
 We consider it as a geometric version of 
 the regularity of the language $L \subseteq S^*$
 defining a normal form. 
 Indeed, if a language 
 $L \subseteq S^*$ is regular, then 
 the normal form it defines is quasiregular. 
 Note that a quasiregular normal form does not necessarily define a regular language 
 $L \subseteq S^*$.   
 \begin{definition}[quasiregular normal form]
 	\label{quasireg_def}	
 	We say that a normal form   
 	defined by a language 
 	$L \subseteq S^*$ is quasiregular 
 	if there exists a constant $c \geqslant 0$ such that for  each prefix $u \in S^*$ of a word $w = uv \in L$  there is a word $x \in S^*$ of length $|x| \leqslant c$ for which $ux \in L$.  
 \end{definition} 
  \noindent We notice that a normal 
  form constructed using the first 
  way in the subsection \ref{two_ways_fellow_prop} 
  is not quasiregular.  
  Indeed, there are infinitely many 
  words $u^k x \in L$, where $u$ is a fixed 
  loop and $|x|\leqslant c$, that represent 
  group elements for which the distances
  to $e \in G$ are bounded by some constant 
  from above. The number of such group 
  elements is finite. Since we consider only 
  one--to--one normal forms, the latter is impossible. 
  Also, there does not seem 
  to be any trivial construction of 
  quasiregular normal forms using the 
  second way in the subsection \ref{two_ways_fellow_prop}.  
   
  One can consider a more restricted version 
  of the notion of a quasiregular normal forms by 
  additionally requiring 
  in Definition \ref{quasireg_def} 
  that $x$ is a prefix of $v$. This leads 
  to the notion of a quasiprefix--closed normal 
  form introduced in Definition 
  \ref{quasiprefclos_def}.    
  \begin{definition}[quasiprefix--closed 
	normal form]
	\label{quasiprefclos_def}	
	We say that a normal form  
	defined by a language 
	$L \subseteq S^*$ is  
	quasiprefix--closed if there 
	exists a constant $C \geqslant 0$ such that for every prefix $u \in S^*$ of a word 
	$w = u v \in L$ there is
	a prefix $x \in S^*$ of length 
	$|x| \leqslant C$ of the word $v$ 
	for which $ux \in L$.    
\end{definition}    
 \noindent Prefix--closed normal forms 
 are exactly quasiprefix--closed normal forms with the constant $C=0$.
 A quasiprefix--closed normal form is  
 quasiregular. 
 The reverse in general is not true. 
 However, Theorem \ref{quasirefequalquasiprefclos}
 shows that if a quasiregular 
 normal form satisfies the 
 $f(n)$--fellow traveler property, then 
 one can construct 
 a quasiprefix--closed normal form 
 which also satisfies the $f(n)$--fellow traveler property. So in the context of the
 $f(n)$--fellow traveler property 
 the notions of a quasiregular normal form and 
 a quasiprefix--closed normal form 
 are equivalent. 
    	 	   
          
 \begin{theorem} 
 \label{quasirefequalquasiprefclos}    
     Suppose $L \subseteq S^*$ defines
     a quasiregular  normal form for some 
     constant $c \geqslant 0$  
     which satisfies the $f(n)$--fellow traveler property. 
     Then one can construct a 
     quasiprefix--closed normal form 
     $L' \subseteq S^*$ for the constant $C=4c$ which also satisfies the $f(n)$--fellow traveler property.   
 \end{theorem}	
 \begin{proof} 
    First we note that if $c = 0$, then the 
    normal form defined by the language 
    $L$ is prefix--closed. So we further  
    assume that $c>0$. 
 	Now let $u$ be a prefix of length $k(c+1)$ of a 
 	word $w = u v \in L$ for an integer 
 	$k \geqslant 0$.   	 
 	Since $L$ is a quasiregular normal form, 
 	there exists $x \in S^*$ 
 	of length $|x| \leqslant c$
 	for which $ux \in L$. 
 	Let $y \in S^*$ be any word of length 
 	$c-|x|$.  
	We define $q(u)$ to be the word 
	$xyy^{-1}x^{-1} \in S^*$, where $x^{-1}$ and 
	$y^{-1}$ are the inverses of $x$ and $y$,  
	respectively. The word $q(u)$ depends on 
	$u$ only. 
	
    A word $w \in L$ can be always written 
    as a concatenation 
    $w = u_1 u_2 \dots u_m t$ of words $u_1,\dots,u_m \in S^*$ and 
    $t \in S^*$, where 
    $|u_1| = \dots = |u_m| = c+1$ and 
    $0 \leqslant |t| < c + 1$.  
    For each word $w \in L$ we construct 
    a word $w'$ as follows: 
    \begin{equation}
    \label{from_w_to_wprime}   
       w' = u_1 q(u_1)  \dots 
            u_{m-1} q(u_1 \dots u_{m-1}) u_m  t.
    \end{equation} 
    We define a language $L'$ 
    as $L' = \{w' \, | \, w \in L\}$.  
    Let us show that $L'$ is a  
    quasiprefix--closed normal form 
    for the constant $C = 4c$.  
    We denote by $x_i,y_i \in S^*$  
    the words for which 
    $q(u_1 \dots u_i) = x_i y_i y_i^{-1}x_i^{-1}$.
    Let $u'$ be a prefix of $w'$.
    There are the following  three possibilities.
    \begin{itemize}
    \item{The prefix $u'$ is of the form 
    $u' = u_1 q(u_1) \dots 
         u_{i-1} q(u_1 \dots u_{i-1}) p_i$, where 
    $i=1,\dots,m-1$ and 
    $p_i$ is a prefix of $u_i$. 
    By the definition of $x_i$, we have 
    that $u_1 \dots u_{i-1}u_i x_i \in L$. 
    Let $u'' = u_1 q(u_1) \dots 
         u_{i-1}q(u_1 \dots u_{i-1}) u_i x_i$. 
    It can be seen that $u'' \in L'$, 
    $u'$ is a prefix of $u''$ and    
    $u''$ is a prefix of $w'$. 
    Let $u_i  = p_i \tau_i$. 
    Then $u'' = u' \tau_i x_i$, 
    so $|u''|-|u'| = |\tau_i| + |x_i| 
        \leqslant \left(c+1\right) + c 
        \leqslant 4c$.}
    \item{The prefix $u'$ is of the form 
    $u' = u_1 q(u_1) \dots u_{i-1}q(u_1\dots u_{i-1})u_ip_i$, where 
    $i =1,\dots,m-1$ and 
    $p_i$ is a prefix 
    of $q(u_1 \dots u_i) = x_i y_i 
    y_i^{-1} x_i ^{-1}$. 
    If $p_i$ is a prefix of $x_i$, then we put 
    $u'' = u_1 q(u_1) \dots u_{i-1} 
     q(u_1 \dots u_{i-1})u_ix_i$.   
    If $p_i$ is not a prefix of $x_i$, 
    then we put $u'' = u_1 q(u_1) \dots u_{i-1} 
    q(u_1 \dots u_{i-1})u_i 
    q(u_1 \dots u_{i})\gamma_{i}$, where 
    $\gamma_{i} = u_{i+1} x_{i+1}$ if 
    $i<m-1$ and $\gamma_{i} = u_m t$ if $i=m-1$. It can be seen that 
    $u'' \in L '$, $u'$ is a prefix of 
    $u''$ and $u''$ is a prefix of $w'$. 
    In the first case when $p_i$ is a prefix 
    of $x_i$, we have that $|u''|-|u'| \leqslant |x_i| \leqslant c$. 
    In the second case when $p_i$ is not a prefix
    of $x_i$, we have that 
    $|u''|-|u'| \leqslant 
     (2c-1) + (c+1) + c =4c$.}
    \item{The prefix $u'$ is of the form 
    $u'= u_1 q(u_1)  \dots 
    u_{m-1} q(u_1 \dots u_{m-1})p_m$, 
    where $p_m$ is a prefix of $u_mt$. Let $u'' =w'$. Then $u'' \in L$, 
    $u'$ is a prefix of $u''$ and $u''$ is 
    a prefix of $w'$. 
    We have that $|u''|-|u'| \leqslant (c+1) +c \leqslant 4c$.}            
    \end{itemize}         
   It is 
   straightforward from \eqref{from_w_to_wprime}
   that if 
   $L$ satisfies the $f(n)$--fellow
   traveler property, then 
   $L'$ also satisfies the 
   $f(n)$--fellow traveler property.   
 \end{proof}

\section{Quasigeodesic Normal Forms}
\label{quasigeodesic_sec}

This section discusses quasigeodesic 
normal forms in the context of 
the $f(n)$--fellow traveler property. 
In the subsection \ref{non-exist-thm}
we show that  groups with the
strongly--super--polynomial Dehn function 
and non--finitely presented groups 
do not admit quasigeodesic normal forms
satisfying the $f(n)$--fellow traveler 
property. In the subsection \ref{cayley_dist_sec} 
we show relation with the notion of 
a Cayley distance function.

\subsection{Non--existence theorems}
\label{non-exist-thm}

This subsection presents  non--existence 
theorems for quasigeodesic normal forms 
satisfying the $f(n)$--fellow traveler property. 
First we consider finitely presented groups. We show that if the Dehn function of a 
group is strongly--super--polynomial (see 
Definition \ref{ssp_def}), then it does not 
admit a quasigeodesic normal form 
satisfying the $f(n)$--fellow traveler 
property (see Theorem \ref{ssp_no_quasidesic_thm}).
Then we consider non--finitely presented 
groups. We show that they do not admit 
quasigeodesic normal forms satisfying the 
$f(n)$--fellow traveler property 
(see Theorem \ref{nonfp_no_quasigeodesic_thm}).

\vskip1mm

{\it Finitely presented groups}. 
Let $G$ be a 
group.
We assume first that $G$ is 
finitely presented.
Let $G =\langle A \, |\, R \rangle$ be a 
finite presentation of $G$.  
We denote by $F_A$ a free group on $A$. 
Let $w \in F_A$ be a reduced word for which 
$\pi(w)=e$ in the group $G$. 
Let us recall the definitions of 
a combinatorial area  and 
the Dehn function.  
\begin{definition}[combinatorial area] 
	The area $\mathcal{A}(w)$ with respect 
	to the presentation 
	$\langle A  \, | \,R \rangle$ 
	is the minimum $N$ for which 
	$w = \prod\limits_{i=1}^{N}    \tau_i^{-1}r_i^{\pm 1}\tau_i$ 
	in $F_A$, where $r_i \in R$ 
	and $\tau_i \in F_A$.
\end{definition}
\begin{definition}[Dehn function] 
	The Dehn function of $G$ with respect to 
	the presentation $\langle A \, | \, R \rangle$ is the function 
	$\delta : \mathbb{N} \rightarrow \mathbb{N}$ 
	such that 
	$\delta (n) = 
	\max \{ \mathcal{A}(w) \,|\, w \in F_A, |w|\leqslant n\}$.  
\end{definition} 

\noindent Now we recall the notion of 
a strongly--super--polynomial function
introduced in \cite{BET22}. 
\begin{definition}[strongly--super--polynomial function]
	\label{ssp_def}	
	A non--zero function 
	$f \in \mathcal{F}$ is 
	said to be 
	strongly--super--polynomial 
	if $n^2 f \ll f$. 
\end{definition}	
\noindent Note that for every $c>0$, one has that 
$n^2 f \ll f$ if and only if $n^c f \ll f$ \cite{BET22}. 
Strongly--super--polynomial 
functions include, for example, exponential functions.

\vskip1mm 

The following lemma is 
a key ingredient in the proof of 
Theorem \ref{ssp_no_quasidesic_thm}.

\begin{lemma} 
\label{dehn_ineq_quasigeodesic}    
   If a group $G = \langle A \, |\, R \rangle$ 
   admits a quasigeodesic 
   normal form satisfying the 
   $f(n)$--fellow traveler property, then 
   there exist constants $C,D>0$ and 
   $n_0 \geqslant 0$ such that
   for the Dehn function $\delta(n)$ 
   of $G$ with respect to the presentation
   $\langle A \, |\, R \rangle$ the inequality 
   $\delta(n) \leqslant D n^2 
   	  \delta (Cf(Cn))$
   holds for all $n \geqslant n_0$.   
\end{lemma}	
\begin{proof} 
   Suppose that $G$ admits a quasigeodesic
   normal form satisfying the $f(n)$--fellow traveler property. 
   Let $L \subseteq S^*$ be a language defining
   such normal form in $G$.   	
   Now let $w = s_1 \dots s_n \in S^*$ be a word for which  $\pi (w) =e$ in the group $G$, where 
   $n = |w|$. 
   For each $i=1,\dots,n-1$, let 
   $u_i \in L$ be the normal form of the 
   group element $g_i = \pi (s_1 \dots s_i)$.
   We first divide a loop  $w$ into subloops: $s_1 u_1^{-1}$, $u_1 s_2 u_2^{-1},\dots,u_{n-2}s_{n-1}u_{n-1}^{-1}$, $u_{n-1}s_n$.   
   
   For a given $i \in \{1,\dots, n-1\}$
   and $j \geqslant 1$ we denote by  
   $u_{i,j} \in S$ the $j$th symbol in the normal form $u_i \in S^*$, if 
   $j \leqslant |u_i|$, and 
   $u_{i,j} = e$, if $j >|u_i|$. 
   In particular, a prefix of $u_i$ of length 
   $j \leqslant |u_i|$ 
   is  $u_{i,1} \dots u_{i,j}$. 
   For a given $i \in \{1,\dots,n-2\}$ 
   and $j \geqslant 1$, let  
   $v_{i,j} \in S^*$ be a shortest path
   connecting   
   $g_{i,j} = \pi(u_{i,1} \dots u_{i,j})$ and 
   $g_{i+1,j} = \pi(u_{i+1,1} \dots u_{i+1,j})$:
   $g_{i,j}  \cdot v_{i,j} =
    g_{i+1,j}$.
   For illustration see Figure \ref{loop}. 
   \begin{figure}
   	\centering 	
   	\includegraphics[width=6.5cm]{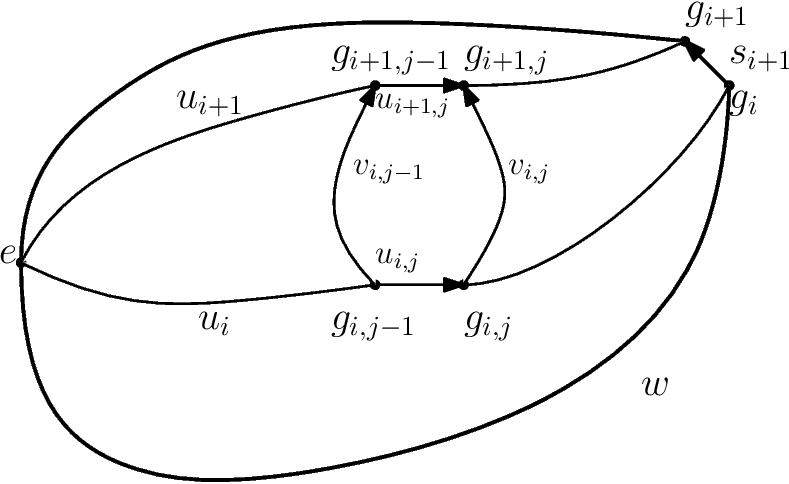}
   	\caption{The outer cycle shows 
    a loop $w = s_1 \dots s_{i} s_{i+1} 
            s_{i+2} \dots s_n$. 
     The directed edge labeled by $s_{i+1}$  
     leads from the group element 
     $g_i$ to 
     the group element 
     $g_{i+1}$.
     The curves $u_{i}$ and 
     $u_{i+1}$ show normal forms 
     of the group elements $g_i$ and 
     $g_{i+1}$, respectively. 
     The curves $v_{i,j}$ and 
     $v_{i,j-1}$ show shortest paths 
     between $g_{i,j}, g_{i+1,j}$ 
     and $g_{i,j-1}, g_{i+1,j-1}$, 
     respectively. 
     The directed edges labeled by 
     $u_{i,j}$ and $u_{i+1,j}$ lead 
     from $g_{i,j-1}$ to 
     $g_{i,j}$ and from $g_{i+1,j-1}$ 
     to $g_{i+1,j}$, respectively.}
   	\label{loop}
   \end{figure}

   For a given $i \in \{ 1,\dots,n-2 \}$, let 
   $m_i = \max \{|u_i|, |u_{i+1}|\} -1$. 
   Now we divide each subloop 
   $u_i s_{i+1} u_{i+1}^{-1}, 
    i = 1, \dots, n-2$ into 
   smaller subloops $w_{i,j}, 
   j=1, \dots, m_i$:  
   if $j=1$, then $w_{i,1} = u_{i,1} v_{i,1} u_{i+1,1}^{-1}$, if 
   $2 \leqslant j \leqslant m_i-1$,
   then $w_{i,j}= u_{i,j} v_{i,j} u_{i+1,j}^{-1} v_{i,j-1}^{-1}$ 
   and if $j=m_i$, then 
   $w_{i,m_i}=u_{i,m_i } s_{i+1}
   u_{i+1,m_i}^{-1}
   v_{i,m_i}^{-1}$.
   Therefore, for the area 
   $\mathcal{A}(w)$ 
   we have:
   $
     \mathcal{A}(w) \leqslant 
     \mathcal{A}(s_1u_1^{-1}) + 
     \mathcal{A}(u_{n-1}s_n)  +
     \sum\limits_{i=1}^{n-2} 
     \sum\limits_{j=1}^{m_i} 
     \mathcal{A}(w_{i,j})$.

   Let $s(n)$ be the function defined 
   by the equation \eqref{def_function_f_1}
   for the normal form $L$ and 
   the set of generators $A$. 
   We have:  
   $|w_{i,j}| \leqslant 
     s(j) + s (j-1) + 2$ for 
    $1 \leqslant j \leqslant m_i -1$ and 
   $|w_{i,m_i}| \leqslant 3 + s(j-1)$
   for $j = m_i$.
   Since the normal form $L$ is 
   quasigeodesic, there exists a 
   positive integer  
   $D$ such that  
   $D n \geqslant m_i$ for all 
   $i=1,\dots, n-2$. 
   
   If $s(n)$ is an unbounded function 
   we may assume that $s(D) \geqslant 
   \max\{|s_1 u_1^{-1}|,|u_{n-1}s_n|\}$; 
   in particular, $s(D) \geqslant 1$.  
   The total number of terms 
   in the expression 
   $\mathcal{A}(s_1u_1^{-1}) + 
   \mathcal{A}(u_{n-1}s_n)  +
   \sum\limits_{i=1}^{n-2} 
   \sum\limits_{j=1}^{m_i} 
   \mathcal{A}(w_{i,j})$ is 
   at most $Dn^2$ and each term is bounded 
   from above by $4s(Dn)$. Therefore, 
   $\mathcal{A}(w) 
    \leqslant Dn^2 \delta (4s(Dn))$ 
   which implies that 
   $\delta(n) \leqslant Dn^2 
    \delta(4 s(Dn))$ for all $n$.
   Since the normal form $L$ satisfies the
   $f(n)$--fellow traveler property, we 
   have that $s \preceq f$ which implies 
   that there exists a constant $C_0>0$ and 
   $n_0 \geqslant 0$ such that 
   $s(n) \leqslant C_0 f(C_0 n)$ for all 
   $n \geqslant n_0$. Therefore, 
   $\delta (n) \leqslant 
   D n^2 \delta (4 s (Dn)) \leqslant 
   D n^2 \delta (4 C_0 f(C_0 D n))$. 
   Let $C= \max\{4C_0, C_0D \}$. 
   Then we have that 
   $\delta (n) \leqslant D n^2 
    \delta (C f(Cn))$ for all 
    $n \geqslant n_0$.  
   
   If the function $s(n)$ is 
   bounded from 
   above by a constant, 
   then we immediately obtain that 
   $\delta(n)= O(n^2)$. 
   In particular, one can  
   always get that 
   $\delta(n) \leqslant D n^2 
    \delta (Cf(Cn))$ 
   for all $n \geqslant n_0$
   for some integer constants $C,D>0$ and 
   $n_0 \geqslant 0$.   
\end{proof}

\begin{theorem}
\label{ssp_no_quasidesic_thm}   
   If the Dehn function of a group 
   $G = \langle A \, |\, R \rangle$ with respect 
   to the presentation $\langle A \, |\, R \rangle$ 
   is strongly--super--polynomial, then $G$ does not admit 
   a quasigeodesic normal form that satisfies the $f(n)$--fellow traveler 
   property.     
\end{theorem}	
\begin{proof}
   Let $\delta(n)$ be the Dehn 
   function of 
   a group $G$ with respect to a presentation
   $\langle A \, | \, R \rangle$. 
   Since the Dehn function $\delta(n)$ is 
   strongly--super--polynomial, 
   $n^2 \delta (n) \ll \delta (n)$.  
   Therefore,  there exist  
   an unbounded function $t(n) \in \mathcal{F}$
   and constants $K, M, N_0$ such that for all $n \geqslant N_0$:    
   \begin{equation}
   \label{thmA_super_ineq}  
      n^2 \delta (n) t(n) \leqslant K \delta (Mn).
   \end{equation}

   We prove the theorem by 
   contradiction. 
   Assume that $G$ admits a 
   quasigeodesic
   normal form satisfying the 
   $f(n)$--fellow traveller property for some 
   function $f \ll \mathfrak{i}$.   
   By Lemma \ref{dehn_ineq_quasigeodesic} 
   there exist 
   integer constants 
   $C,D > 0$ and $N_1 \geqslant 0$ such that
   for all $n \geqslant N_1$: 
   \begin{equation} 
   \label{thmA_prop_ineq}  
     \delta (n) \leqslant D n^2 \delta (Cf(Cn)).
   \end{equation}   
   Let $N_2 = \max \{N_0, N_1\}$. 
   By the inequalities \eqref{thmA_super_ineq} and 
   \eqref{thmA_prop_ineq} we obtain that 
   for all $n \geqslant N_2$:  
   $$
     n^2 \delta(n) t(n) \leqslant K \delta (M n) 
     \leqslant D K M^2 n^2 \delta  (Cf (CMn)). 
   $$ 
   Therefore, 
   $\delta(n) t(n) 
    \leqslant  D K M^2 \delta (Cf(CMn))$ 
   for all $n \geqslant N_2$.  
   Let $N_3 = \min \{ n \, | \, 
        2 DKM^2 \leqslant t(n)\}$ and
       $N_4 = \max \{ N_2, N_3\}$.   
   Then, $2 \delta (n) \leqslant \delta (Cf(CMn)) $ 
   for all $n \geqslant N_4$. 
   
   Let $N_5 =  \min \{ n \, | \, \delta (n) \geqslant 1\}$
   and $N_6 = \max \{N_4, N_5\}$. 
   Then, $n \leqslant C f (CMn)$ for all $n \geqslant N_6$ as if, otherwise, 
   $n > Cf (CMn)$ for some $n \geqslant N_6$,  then $2 \delta (n) \leqslant \delta(Cf(CMn)) \leqslant \delta (n)$  which leads to a contradiction with the inequality 
   $\delta (n) \geqslant 1$. 
   Since $f \ll \mathfrak{i}$, there exists 
   an unbounded function $\tau (n) \in \mathcal{F}$ 
   and constants $E,N_7$ such that  
   $f(CMn)\tau(CMn) \leqslant En$ 
   for all $n \geqslant N_7$. 
   Let $N_8 = \max \{N_6, N_7\}$. From the inequalities 
   $n \leqslant Cf (CMn)$ and $f(CMn)\tau(CMn) \leqslant En$ 
   we conclude that 
   $n \tau (CMn) \leqslant C E n$ for all 
   $n \geqslant N_8$. As the function 
   $\tau (n)$ is unbounded, we get a contradiction. 
\end{proof}

{\it Non--finitely presented groups}. 
Now we assume that $G$ is 
a non--finitely presented group.  
In this case no additional assumptions 
are needed to show     
non--existence Theorem \ref{nonfp_no_quasigeodesic_thm}. 
Equivalently, Theorem \ref{nonfp_no_quasigeodesic_thm} 
claims that if $G$ is a finitely generated group
admitting a quasigeodesic normal form with the 
$f(n)$--fellow traveler property, then $G$ is
finitely presented.
\begin{theorem} 
\label{nonfp_no_quasigeodesic_thm}   
   If $G$ is a non--finitely 
   presented group, then 
   $G$ does not admit a 
   quasigeodesic normal form 
   that satisfies the 
   $f(n)$--fellow traveller property.  
\end{theorem}	

\begin{proof} 
   First we notice that there exist 
   infinitely many words $w \in S^*$
   for which $\pi(w)=e$  
   such that for every decomposition of $w$ 
   as the product 
   $w = \prod\limits_{j=1}^{k}r_j^{-1}u_j r_j$, 
   where $r_j,u_j \in S^*$ and 
   $\pi(u_j) = e$ for  $j=1,\dots,k$,
   for some  $1 \leqslant m \leqslant k$ 
   the length of $u_m$ is greater than 
   or equal to the length of $w$: 
   $|u_m| \geqslant |w|$.
   Indeed, if such infinitely many words
   did not exist, the group $G$ would be 
   finitely presented. 
   We denote the set of such words 
   by $W$.

   We prove the theorem by contradiction. 
   Assume that $G$ admits a 
   quasigeodesic  normal form satisfying
   the $f(n)$--fellow traveller property  
   for some function $f \ll \mathfrak{i}$. 
   Let $L \subseteq S^*$ be a language 
   defining such normal form in $G$. 
   Let $w = s_1 \dots s_n$ be a word from the 
   set $W$. 
   For each $i=1,\dots,n-1$, let 
   $u_i \in L$ be the normal form of the 
   group element $g_i = \pi (s_1 \dots s_i)$. 
   Similarly to Lemma 
   \ref{dehn_ineq_quasigeodesic},
   we divide a loop $w$ into subloops: 
   $s_1 u_1^{-1}$, 
   $u_1 s_2 u_2^{-1}$,
   \dots, 
   $u_{n-2}s_{n-1}u_{n-1}^{-1}$, 
   $u_{n-1}s_n$. 
   For a given 
   $i \in \{1, \dots, n-1 \}$ and 
   $j \geqslant 1$ we denote by
   $u_{i,j} \in S$ the $j$th symbol 
   in the normal form $u_j \in S^*$, 
   if $j \geqslant |u_i|$, and 
   $u_{i,j}=e$, if $j > |u_i|$. 
   For a given $i \in \{1,\dots,n-2\}$ and 
   $j \geqslant 1$, we denote by 
   $v_{i,j} \in S^{*}$ a shortest path 
   connecting 
   $g_{i,j} = \pi (u_{i,1} \dots u_{i,j})$ and
   $g_{i+1,j} = \pi(u_{i+1,1} \dots u_{i+1,j})$.  
   
   For a given $i \in \{1, \dots, n-2\}$, let
   $\ell_i = \max \{|u_i|, |u_{i+1}|\} -1$. 
   Similarly to the proof of 
   Lemma \ref{dehn_ineq_quasigeodesic},  
   let us divide each subloop 
   $u_i s_{i+1} u_{i+1}^{-1}$, 
   $i=1,\dots,n-2$ into smaller subloops 
   $w_{i,j},j=1,\dots,\ell_i$: if $j=1$, 
   then $w_{i,1} = u_{i,1}v_{i,1} u_{i+1,1}^{-1}$,
   if $2 \leqslant j \leqslant \ell_{i}-1$, then 
   $w_{i,j}=u_{i,j} v_{i,j} u_{i+1,j}^{-1} v_{i,j-1}^{-1}$ and if $j = \ell_i$, 
   then $w_{i,\ell_i}=
         u_{i,\ell_i}s_{i+1}u_{i+1,\ell_i}^{-1}
         v_{i,\ell_i}^{-1}$.

   Let $s(n)$ be the function defined 
   by the equation \eqref{def_function_f_1}
   for the normal form $L$ and 
   the set of generators $A$.
   Let $\ell = \max \{\ell_i \, | \, i =1, \dots,n-2 \}$. 
   Then, $|w_{i,j}| \leqslant 4 s(\ell)$ for all
   $i=1,\dots,n-2$ and $1 \leqslant j \leqslant \ell_i$, 
   where we assume that 
   $\ell$ is big enough so   
   $s(\ell) \geqslant 1$. 
   Since the normal form is quasigeodesic,
   there exists an 
   integer constant 
   $C>0$ for which $\ell_i \leqslant C n$ 
   for all $i=1,\dots, n-2$, so 
   $\ell \leqslant Cn$.   
            
   As $s \preceq f$ for some $f \ll \mathfrak{i}$, there exists an unbounded function $t \in \mathcal{F}$ for which  $s (n) t(n) \preceq \mathfrak{i}$. 
   In particular, $s(n) \leqslant \frac{1}{8C} n$ for all $n \geqslant N$ for some $N$. 
   Therefore, if $n \geqslant N$, then $|w_{i,j}| \leqslant \frac{n}{2}$
   for all  $i=1,\dots,n-2$ and 
   $1 \leqslant j \leqslant \ell_i$.   
   Furthermore, we may assume that $n$ is big enough 
   so $|s_1 u_1^{-1}|,|u_{n-1}s_n| \leqslant \frac{n}{2}$. Since $\frac{n}{2}<n$ and $w \in W$, 
   we get a contradiction.           
\end{proof}

 \subsection{Relation with a Cayley 
 	         distance function}
 \label{cayley_dist_sec}

  We find relation with the notion of 
  a Cayley distance function studied in 
  \cite{BET22,eastwest19,LATA18}. 
  A Cayley distance function 
  $h : \mathbb{N} \rightarrow 
       \mathbb{R}_{+}$ 
  is defined for an arbitrary bijection 
  $\psi : L \rightarrow G$ between 
  a language $L \subseteq S^*$ and 
  a group $G$ by the following identity: 
  \begin{equation*} 
  	 h(n) = \max \{  d_A (\psi(w),\pi(w)) 
  	 \,|\,  w \in L^{\leqslant n} \} \,\, 
  	 \mathrm{if} \,\, 
  	 L^{\leqslant n} \neq \varnothing,   
  \end{equation*} 	
  where 
  $L^{\leqslant n}  = 
   \{ w \, | \, |w| \leqslant n \}$ 
   is the set of words in $L$ 
   of length less than or equal to $n$,  
   and $h(n)=0$ if $L^{\leqslant n} = \varnothing$.  
  
  Cayley automatic groups were introduced 
  by Kharlampovich, Khoussainov and Miasnikov
  \cite{KKM11}. 
  We recall that a group $G$ is called Cayley automatic if there exists a bijection 
  $\psi : L \rightarrow G$ for which $L$ is 
  a regular language and for each 
  $a \in A$ the relation 
  $R_a = \{(u_1,u_2) \in L \times L \,| \, 
   \psi (u_1) a = \psi(u_2) \}$ 
   is recognized by 
  a two--tape synchronous automaton. 
  The bijection 
  $\psi : L \rightarrow G$ is referred to as a 
  Cayley automatic representation of $G$. 
  Cayley automatic groups extend automatic
  groups retaining exactly the same 
  computational model but allowing 
  an  arbitrary bijection $\psi: L 
  \rightarrow G$, 
  not only a canonical mapping 
  $\pi : L \rightarrow G$ like 
  in the notion of an automatic group.    
  
  In \cite{LATA18} it is asked if there 
  exists a Cayley automatic representation
  of a non--automatic group $G$    
  such that for the Cayley distance 
  function
  $h : \mathbb{N} \rightarrow \mathbb{R}_+$ 
  the inequality $h \prec \mathfrak{i}$ holds. 
  This problem can be slightly narrowed by 
  requiring that $h \ll \mathfrak{i}$. 
  Theorem \ref{relation_cayley_automatic_thm}
  shows that if such Cayley automatic 
  representation exists, then
  $G$ admits a quasigeodesic normal 
  form satisfying the $h(n)$--fellow 
  traveler property.

  \begin{theorem}
  \label{relation_cayley_automatic_thm} 	 
 	If a non--automatic group $G$ has a 
 	Cayley automatic representation 
 	$\psi : L \rightarrow G$ with the Cayley distance function 
 	$h \ll \mathfrak{i}$, 
 	then there exists a quasigeodesic normal form 
 	$L' \subseteq S^*$ that satisfies the $h(n)$--fellow traveler property.  
 \end{theorem}	  
 \begin{proof} 
 	First we describe how to construct 
 	a normal form $L'$ from a given Cayley automatic
 	representation $\psi : L \rightarrow G$. 
 	For a given $u \in L$ let $v \in S^* $ be a 
 	word corresponding to a shortest path 
 	between $\pi(u)$ and $\psi(u)$
 	in the Cayley graph 
 	$\Gamma(G,A)$
    for which $\psi(u) = \pi(uv)$. 
    We define the language 
    $L' \subseteq S^*$ as 
    $L' = \{uv \, | \, u \in L \}$.   
    Below we prove that 
    $L'$ defines a quasigeodesic 
    normal form that satisfies the $h(n)$--fellow 
    traveler property. 
     
    We now show that the normal 
    form defined by $L'$ 
    is quasigeodesic. 
    Let $g = \psi(u)$ for some $u \in L$ and $w = uv \in L'$ 
    be the normal form of $g$. 
    We need to show that $|uv| \leqslant C d_A (g) + C$ 
    for some constant $C>0$. 
    Since $\psi : L \rightarrow G$ is a Cayley 
    automatic representation, there exists a 
    constant $C_1>0$ for which 
    $|u| \leqslant C_1 (d_A (g) + 1)$. 
    Let $n = |u|$. Then $|v| \leqslant h(n)$. 
    Therefore, $|w|= |u| + |v| \leqslant n + h(n)$. 
    Since $h \ll \mathfrak{i}$, then 
    $h(n) \leqslant C_2 n$ for some constant $C_2 >0$
    and all $n \geqslant N_0$. Let 
    $C_3  = C_2+1$. 
    Then we have that $|w| \leqslant 
    C_3 n \leqslant C_3 C_1(d_A(g)+1)$ for all 
    $n \geqslant N_0$. Therefore, there exists 
    a constant $C>0$ such that 
    $|w| \leqslant C (d_A (g)+1)$ for all 
    $n \geqslant 0$.
    
    \begin{figure}
    \centering 	
    \includegraphics[width=6.5cm]{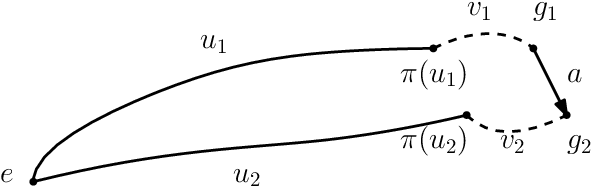}
    \caption{The curves labeled 
    	    $u_1$ and $u_2$ 
    	    correspond to the words 
    	    $u_1,u_2 \in L$ 
            for which $\psi(u_1)= g_1$ and 
            $\psi(u_2) = g_2$. The dashed 
            curves $v_1$ and $v_2$ show shortest paths between 
            $\pi(u_1)$ and $g_1$ and 
            $\pi(u_2)$ and $g_2$, 
            respectively. The directed 
            edge labeled by $a$ leads from 
            $g_1$ to $g_2$.}
    	\label{cayleydist_fig}
    \end{figure}

    Let us show that $L'$ defines a normal 
    form in $G$ that satisfies the $h(n)$--fellow 
    traveler property. 
    Let $w_1 = u_1 v_1 \in L'$ and 
        $w_2 = u_2 v_2 \in L'$ be 
    the normal forms of the 
    group elements 
    $g_1 = \pi (w_1)$ and $g_2 = \pi (w_2)$ 
    for which $g_1 a = g_2$ for some $a \in A$. 
    See Figure \ref{cayleydist_fig} for 
    illustration. 
    We denote by $m$ 
    the minimum $m = \min \{|u_1|, |u_2|\}$.
    Let $s(n)$ be the function defined 
    by the equation \eqref{def_function_f_1}
    for the normal form $L'$ and 
    the set of generators $A$. 
    It can be shown that 
    $s(n) \leqslant 2 h(n) + C_0$ 
    for all $n \leqslant m$ and some constant 
    $C_0>0$; 
    this is proved in details
    for all $n \geqslant 0$ 
    in \cite[Theorem~2.1]{eastwest19}. 
    If $n \geqslant m$, then by the 
    triangle inequality  
    $d_A (\pi (w_1(n)),\pi(w_2(n))) 
     \leqslant |v_1| + |v_2|  + ||u_1|-|u_2|| + 1$. 
    Since the representation 
    $\psi : L \rightarrow G$ is Cayley automatic, 
    $||u_1|-|u_2|| \leqslant C_2$ for some constant 
    $C_2$. Let $C_3 = C_2 + 1$. 
    Therefore, for all $n \geqslant m$: 
    \begin{equation*}
    \begin{split}
    d_A (\pi (w_1(n)),\pi(w_2(n))) 
    \leqslant |v_1| + |v_2|  + C_3 \leqslant 
    h(|u_1|) + h(|u_2|) + C_3 \leqslant \\ 
    h(m) + h (m+C_2) + C_3 \leqslant 2 h(m + C_2) + C_3 \leqslant 2 h (n+ C_2) + C_3.
    \end{split}
    \end{equation*} 
    Let $C_4 = \max \{C_0,C_3\}$. 
    Then, for all $n \geqslant 0$, we have that 
    $d_A (\pi (w_1(n)),\pi(w_2(n))) 
     \leqslant 2 h (n + C_2) + C_4$. 
    Let $f(n)= 2 h (n + C_2) +  C_4$.
    Since $f (n) \preceq h(n)$, 
    we get that 
    $s(n) \preceq h(n)$. 
    Therefore, $L'$ defines a normal form 
    of $G$
    satisfying the $h(n)$--fellow traveler 
    property. 
 \end{proof}

 \section{Quasiregular Normal Forms} 
 \label{quasiregular_sec}
 
 This section discusses quasiregular normal 
 forms in the context of the $f(n)$--fellow 
 traveler property. 
 By Theorems 
 \ref{ssp_no_quasidesic_thm} and  
 \ref{nonfp_no_quasigeodesic_thm} 
 for groups with the 
 strongly--super--polynomial 
 Dehn function and
 non--finitely presented groups
 there exist no quasigeodesic normal forms satisfying the
 $f(n)$--fellow traveler property. 
 In this 
 section we show  
 examples of quasiregular 
 normal forms 
 satisfying the $f(n)$--fellow traveler 
 property for such groups. 
 
 \subsection{Baumslag--Solitar groups} 
 
 We consider a family of 
 the Baumslag--Solitar groups 
 $BS(p,q) = \langle a, t \, 
 | \, t a^p t^{-1} = a^q \rangle$ for 
 $1 \leqslant p < q$. Each group of this family 
 has the exponential Dehn function, so by 
 Theorem \ref{ssp_no_quasidesic_thm} it does not 
 admit a quasigeodesic normal form satisfying 
 the
 $f(n)$--fellow traveler property. In Theorem \ref{BS_qusireg_normal_form} we will show that 
 each group of this family admits a quasiregular
 normal form satisfying 
 the $\log(n)$--fellow 
 traveler property.

 \begin{theorem}
 \label{BS_qusireg_normal_form}   
    Each group $BS(p,q)$ for 
    $1 \leqslant p < q$ 
    admits a quasiregular normal form satisfying the 
    $\log(n)$--fellow traveler property.  
 \end{theorem}
 \begin{proof}   
    Every group element  $ g \in BS(p,q)$ 
    for $1 \leqslant p < q$
    can  be uniquely written as a  
    freely reduced word over the alphabet  
    $\{a^{\pm 1}, t^{\pm 1}\}$
    of the form: 
    \begin{equation}
    	\label{normal_form_BS}   
    	w = w_\ell t^{\varepsilon_\ell} \dots
    	w_1 t^{\varepsilon_1} a^k,    
    \end{equation}
    where $\varepsilon_i \in \{+1,-1\}$,   
    $w_i \in \{\epsilon, a, \dots, a^{p-1}\}$  
    if $\varepsilon_i  = -1$,
    $w_i \in \{\epsilon, a, \dots, a^{q-1}\}$ 
    if $\varepsilon_i = +1$ and $k \in \mathbb{Z}$. 
    So the identity \eqref{normal_form_BS} 
    defines a normal form in $BS(p,q)$.  
    This normal form 
    is prefix--closed, so it is quasiregular. 
    Below we show that it satisfies the 
    $\log(n)$--fellow traveler property. 
    
    Let $A = \{a,t\}$ and $d_A$ be the word metric 
    in $BS(p,q)$ with respect to the generators 
    $a$ and $t$. 
    Burillo and Elder \cite{BurilloElder14}
    showed that there exist 
    constants $C_1,C_2,D_1,D_2 > 0$ such that 
    for all $g \in BS(p,q)$: 
    \begin{equation} 
    \label{Elder_Burillo_estimate}	
       C_1 (\ell +\log (|k|+1)) - D_1 
       \leqslant d_{A} (g)   \leqslant
       C_2 (\ell +\log (|k|+1)) + D_2.
    \end{equation}	  
    We first consider a pair of group 
    elements $g$ and $ga$.  
    Let $w_a$ be the normal form of $ga$. 
    Then $w_a = w_\ell t^{\varepsilon_\ell} \dots
    w_1 t^{\varepsilon_1} a^k$. Clearly, 
    $d_A (\pi(w(n)), \pi(w_a(n))) \leqslant 1$ for all $n \geqslant 0$. 
    
    Now we consider a pair of group elements  
    $g$ and $gt$. We denote by $w_t$ 
    the normal form of $gt$. 
    Let $k = mq + r$, where $m \in \mathbb{Z}$ and
    $r \in \{0, \dots, q-1 \}$.  
    We have three different cases: 
    \begin{itemize}
    \item{Suppose $r \neq 0$. Then 
          $w_t = w_\ell t^{\varepsilon_\ell} \dots
          w_1 t^{\varepsilon_1} a^r t a^{mp}$.
          Let $u = w_\ell t^{\varepsilon_\ell} \dots
          w_1 t^{\varepsilon_1}$. 
          If $n \leqslant |u|$, then  
          $d_A (\pi (w(n)), \pi (w_t(n)))=0$. 
          If $|u| < n \leqslant |u|+r+1$, 
          then $d_A (\pi(w(n)), \pi (w_t(n)))
           \leqslant 2 r +2$. 
          If $n > |u|+r+1$, then 
          $d_A (\pi (w(n)), \pi(w_t(n))) 
          \leqslant d_A (a^rta^{i}) + 
                    d_A (a^j)$, where 
           $i = \min \{|mp|,n-(|u|+r+1)\} * 
           \mathrm{sign}(k)$ and 
           $j = \min \{|k|, n - |u|\} 
           * \mathrm{sign}(k)$.
           By \eqref{Elder_Burillo_estimate}, 
           $d_A(a^rta^{i}) \leqslant 
            C_2 (1 + \log (|i|+1)) + D_2$ and 
           $d_A (a^j) \leqslant C_2 
           \log (|j|+1) + D_2$. 
           Therefore,    
           $d_A (\pi (w(n)),\pi (w_t(n)))
            \leqslant 2C_2 \log(n+1) + C_2 + 2D_2$.}
     \item{Suppose $r=0$ and either 
           $\varepsilon_1 = 1$ or $\ell = 0$.  
           Then 
           $w_t = w_\ell t^{\varepsilon_\ell} \dots
           w_1 t^{\varepsilon_1}  t a^{mp}$.
           Let $u = w_\ell t^{\varepsilon_\ell} \dots
           w_1 t^{\varepsilon_1}$.
           If $n \leqslant |u|$, then 
           $d_A (\pi(w(n)), \pi (w_t(n)))=0$. 
           If $n > |u|$, then 
           $d_A (\pi(w(n)), \pi (w_t(n))) 
            \leqslant 
            d_A (ta^{i}) + 
            d_A (a^j)$, where 
            $i = \min \{|mp|,n-|u| - 1\} * 
            \mathrm{sign}(k)$ and 
            $j = \min \{|k|, n - |u|\} 
            * \mathrm{sign}(k)$.
            By \eqref{Elder_Burillo_estimate}, 
            $d_A(ta^{i}) \leqslant 
            C_2 (1 + \log (|i|+1)) + D_2$ and 
            $d_A (a^j) \leqslant C_2 
            \log (|j|+1) + D_2$.
            Therefore,     
            $d_A (\pi (w(n)),\pi (w_t(n)))
            \leqslant 2C_2 \log(n+1) + C_2 + 2D_2$.}
      \item{Suppose $r=0$, $\varepsilon_1 = -1$ 
      	    and $\ell \geqslant 1$. 
      	    Then $w_t =  w_\ell t^{\varepsilon_\ell} \dots
      	    w_2 t^{\varepsilon_2} w_1  a^{mp}$. 
      	    Let $v = w_\ell t^{\varepsilon_\ell} \dots
      	    w_2 t^{\varepsilon_2} w_1$.
      	    If $n \leqslant |v|$, then 
      	    $d_A (\pi (w(n)), \pi (w_t(n)))=0$.
      	    If $n > |v|$, then 
      	    $d_A (\pi (w(n)), \pi (w_t(n)))
      	    \leqslant d_A (a^i) + d_A (t^{-1}a^j)$, 
      	    where 
      	    $i = \min \{|mp|,n-|v|\} * 
      	    \mathrm{sign}(k)$ and 
      	    $j = \min \{|k|, n - |v| -1\} 
      	    * \mathrm{sign}(k)$. 
      	    By \eqref{Elder_Burillo_estimate}, 
      	    $d_A(a^{i}) \leqslant 
      	    C_2 \log (|i|+1) + D_2$ and 
      	    $d_A (t^{-1}a^j) \leqslant C_2 
      	    (1 + \log (|j|+1)) + D_2$.
      	    Therefore,     
      	    $d_A (\pi (w(n)),\pi (w_t(n)))
      	    \leqslant 2C_2 \log(n+1) + C_2 + 2D_2$.}                     
    \end{itemize}	      
    From these three cases we can see that 
    $s(n) \preceq \log (n)$. 
    Therefore, the normal form 
    given by the identity \eqref{normal_form_BS}
    satisfies the $\log (n)$--fellow traveler 
    property.  
 \end{proof}	
 \begin{remark}
    We note that for the normal form 
    in the proof of Theorem \ref{BS_qusireg_normal_form} 
    the upper bound $s(n) \preceq \log (n)$ is 
    sharp. Indeed, let $w = a^{mq^2}$ for $m>0$. Then 
    $w_t = ta^{mpq}$. 
    For $n = mq$
    we have that 
    $d_A (\pi(w(n)),\pi(w_t(n))) =
    d_A (a^{mq},ta^{mq-1}) = 
    d_A (a^{-mq}ta^{mq-1}) = 
    d_A (ta^{m(q-p)-1})$. 
    By \eqref{Elder_Burillo_estimate},
    $d_A (ta^{m(q-p)-1}) \geqslant 
    C_1 (1 + \log |m(q-p)|) - D_1 = 
    C_1 \left(1 - \log \left( \frac{q}{q-p} \right) + \log (mq) 
    \right) - D_1$.
    Therefore, 
    for $n =mq$ we have that:
    $$d_A (\pi(w(n)),\pi(w_t(n))) 
     \geqslant C_1 \log (n) - \left(D_1 + C_1 \log \left(\frac{q}{q-p}\right) -C_1\right).$$
     Therefore,  
     by the triangle inequality we get that for all $n$:
     \begin{equation*}
     d_A (\pi(w(n)),\pi(w_t(n))) 
     \geqslant C_1 \log (n) - \left(D_1 + C_1 \log \left(\frac{q}{q-p}\right) -C_1 + 2q\right).
     \end{equation*}
      This implies that 
      $\log(n) \preceq s(n)$. 
 \end{remark} 

 \begin{remark} 
 	Let us be given a normal form 
    of $BS(p,q)$ for $1 \leqslant p < q$
 	satisfying the $f(n)$--fellow traveler property. 
 	We denote by $m_{a}, m_{a^{-1}}, 
 	m_{t}$ and $m_{t^{-1}}$ the functions 
 	which send the normal 
 	form of a group element 
 	$g \in BS(p,q)$ to 
 	the normal form of a group element 
 	$ga,ga^{-1},gt$ and $gt^{-1}$, 
 	respectively.  
 	One can notice that
 	the functions $m_a, m_{a^{-1}}, m_t$ and 
 	$m_{t^{-1}}$ cannot be all computed 
    in 	$o(n \log n)$ time
    on a one--tape Turing machine.      
    Indeed, suppose each of the 
    functions $m_a, m_{a^{-1}}, 
    m_t$ and $m_{t^{-1}}$ is computed 
    on a one--tape Turing machine
    in $o(n \log n)$ time.    
    Hartmanis \cite{Hartmanis68} 
    and, independently, 
    Trachtenbrot \cite{Trachtenbrot64} 
    showed that a 
    language recognized on a 
    one--tape Turing machine in  
    $o(n \log n)$ time must be regular. 
    This fact and the pumping lemma
    imply that if each of 
    the functions $m_a,m_{a^{-1}}, 
    m_t$ and $m_{t^{-1}}$ is computed
    on a one--tape Turing machine in 
    $o(n \log n)$ time, then 
    the normal form satisfies 
    the bounded length difference 
    property, so it is quasigeodesic;  
    for the proof see \cite[Theorem~4]{KB23}. 
    Thus, by Theorem \ref{ssp_no_quasidesic_thm}, 
    we arrive at a contradiction. 
    However, it can be verified that
    for the normal form given by 
    the identity \eqref{normal_form_BS}   
    the functions $m_a,m_{a^{-1}},m_t$
    and $m_{t^{-1}}$ can be computed
    in $O(n)$ time using
    a more powerful computational model  
    -- a two--tape Turing machine.
    Though this verification is not 
    difficult we omit it as it is 
    out of scope of this paper.
 \end{remark}	
 
 \subsection{Wreath product  
 	            $\mathbb{Z}_2 \wr \mathbb{Z}^2$} 
 
 We consider the wreath product 
 $\mathbb{Z}_2 \wr \mathbb{Z}^2 = 
 \langle a, b , c \, | \, 
 [a^{i_1} b^{j_1} c a^{-i_1} b^{-j_1},   a^{i_2} b^{j_2} c a^{-i_2} b^{-j_2}] =e, ab=ba, c^2 =e \rangle$. This group is  
 non--finitely presented, so 
 by Theorem \ref{nonfp_no_quasigeodesic_thm} 
 it does not admit a quasigeodesic 
 normal form satisfying the 
 $f(n)$--fellow traveler property. 
 In Theorem \ref{Z2wrZ2_quasireg_norm_form}
 we will show that 
 $\mathbb{Z}_2 \wr \mathbb{Z}^2$ 
 admits a quasiregular normal 
 form satisfying the $\sqrt{n}$--fellow 
 traveler property.  
 
 Every group element of 
 $\mathbb{Z}_2 \wr \mathbb{Z}^2$
 can be written as a pair 
 $(\varphi, z)$, where 
 $\varphi : \mathbb{Z}^2 \rightarrow \mathbb{Z}_2$ is a function 
 such that $\varphi(\xi)$ is not equal to  
 the identity for at most finitely many 
 $\xi \in \mathbb{Z}^2$ and $z \in \mathbb{Z}^2$. 
 For the group 
 $\mathbb{Z}^2 = \{(x,y)\, | \, 
 x,y \in \mathbb{Z} \}$ we denote by 
 $a$ and $b$ the generators 
 $a = (1,0)$ and $b = (0,1)$. 
 The group $\mathbb{Z}^2$ is canonically
 embedded in 
 $\mathbb{Z}_2 \wr \mathbb{Z}^2$
 by mapping $\xi \in \mathbb{Z}^2$ 
 to $(\varphi_e,\xi)$, where 
 $\varphi_e : \mathbb{Z}^2 \rightarrow   
  \mathbb{Z}_2$ 
 sends every element 
 of $\mathbb{Z}^2$ to the identity 
 $e \in \mathbb{Z}_2$. 
 We denote by $c$ the nontrivial 
 element of $\mathbb{Z}_2$. 
 The group $\mathbb{Z}_2$ is canonically 
 embedded in 
 $\mathbb{Z}_2 \wr \mathbb{Z}^2$ 
 by mapping $c$ to 
 $(\varphi_c,(0,0))$, where  
 $\varphi_c : \mathbb{Z}^2 
  \rightarrow \mathbb{Z}_2$ is a function
 such that $\varphi_c(\xi)=e$ if 
 $\xi \neq (0,0)$ and 
 $\varphi_c ((0,0))=c$.
 We will identify $a,b$ and $c$ with the 
 group elements 
 $(\varphi_e,a)$, $(\varphi_e,b)$ and
 $(\varphi_c, (0,0))$  in 
 $\mathbb{Z}_2 \wr \mathbb{Z}^2$, 
 respectively.
 The set $A = \{a,b,c\}$
 generates the group 
 $\mathbb{Z}_2 \wr \mathbb{Z}^2$.

 \begin{theorem}
 \label{Z2wrZ2_quasireg_norm_form}	 
 	The wreath product
 	$\mathbb{Z}_2 \wr \mathbb{Z}^2$ 
 	admits a quasiregular normal form 
 	satisfying the $\sqrt{n}$--fellow traveler 
 	property.  	
 \end{theorem}	
 \begin{proof} 
    Let  $\Gamma$ be the infinite directed graph
    shown in Fig.~\ref{z2wrz2fig1}
    which is isomorphic to $\left(\mathbb{N};\mathrm{S}\right)$, 
    where $\mathrm{S}$ is the successor 
    function $\mathrm{S}(n) = n + 1$. 
    The vertices of $\Gamma$
    are identified with elements of $\mathbb{Z}^2$,    
    each vertex  
    in $V(\Gamma) \setminus \{(0,0)\}$ has 
    exactly one ingoing and one outgoing edges and the vertex $(0,0)$ has one outgoing 
    edge and no ingoing edges.  
    Let  $\tau : \mathbb{N} \rightarrow \mathbb{Z}^2$ 
    be the mapping such that
    $\tau(0) = (0,0)$ and, for $k>0$, 
    $\tau(k) = (x,y)$ 
    is the end vertex of a directed path in 
    $\Gamma$ of length $k$ which starts 
    in the vertex $(0,0)$.
    
    \begin{figure}
    	\centering 	
    	\includegraphics[width=4cm]{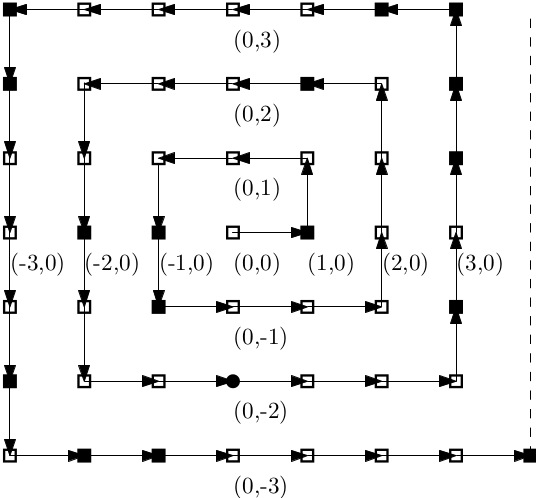}
    	\caption{An infinite digraph $\Gamma$ and 
    		an element $h \in \mathbb{Z}_2 \wr \mathbb{Z}^2$.}
    	\label{z2wrz2fig1}
    \end{figure}
    
    Now we define a normal form in the 
    group $\mathbb{Z}_2 \wr \mathbb{Z}^2$ 
    as follows. 
    Let $g = (\varphi,z) \in \mathbb{Z}_2 \wr 
         \mathbb{Z}^2$, where 
         $\varphi : \mathbb{Z}^2 
          \rightarrow \mathbb{Z}_2$ 
         and $z \in \mathbb{Z}^2$.  
    Let $u_0 = c$ if 
    $\varphi(\tau(0))=c$ and 
    $u_0 = \varepsilon$ if 
    $\varphi(\tau(0))=e$. 
    We denote by $m$ the maximum 
    $m = \max \{i \,|\, \varphi(\tau(i))=c\}$; 
    if $\varphi (\xi) = e$ for all $\xi$ then 
    we put $m=0$. 
    For a given integer 
    $i \in \left[1, m \right]$, 
    let $\alpha_i = a,b,a^{-1}$ 
    and $b^{-1}$ if 
    $\tau(i) - \tau(i-1)$ is equal to 
    $(1,0)$, $(0,1)$, $(-1,0)$ and 
    $(0,-1)$, respectively, 
    and    
    let $\beta_i = c$ if 
    $\varphi(\tau(i))=c$ and 
    $\beta_i = \varepsilon$ if 
    $\varphi(\tau(i))=e$. 
    Let $u_i = 
    \alpha_i \beta_i$ for 
    $i \in \left[1, m \right]$ and 
    $u$ be the concatenation 
    $u=u_0 u_1 \dots u_m$. 
    
    Let $l$ be the integer for which 
    $\tau(l) = z$.  
    If $l > m$, for a given 
    $i \in \left[1, l - m \right]$ 
    let $v_i = a, b, a^{-1}$ and 
    $b^{-1}$ if 
    $\tau (m + i) - \tau (m + i -1)$ is 
    equal to $(1,0)$, $(0,1)$, $(-1,0)$ and 
    $(0,-1)$, respectively.    
    If $l < m$, for a given 
    $i \in \left[1, m- l \right]$ 
    let $v_i = a, b, a^{-1}$ and 
    $b^{-1}$ if 
    $\tau (m - i) - \tau (m -i +1)$ is equal 
    to  $(1,0)$, $(0,1)$, $(-1,0)$ and 
    $(0,-1)$, respectively.
    If $l = m$, let 
    $v = \varepsilon$. If $l \neq m$, 
    let $v = v_1 \dots v_{k}$, where 
    $k = |l-m|$. 
    
   Finally we define a normal form $w$ of the 
   element $g = (\varphi,z)$ as
   a concatenation of $u$ and $v$: 
   $w = uv$. Informally speaking, 
   the normal form $w$ is obtained 
   as follows. 
   Imagine the lamplighter who
   moves along the graph $\Gamma$ 
   starting from $(0,0)$ writing 
   $a,a^{-1},b$ and $b^{-1}$ depending
   whether it moves right, left, up and down, 
   respectively. The lamplighter writes $c$ if 
   the lamp at the current position 
   $(x,y) \in \mathbb{Z}^2$ is 
   lit: $\varphi((x,y))= c$. After 
   the lamplighter reaches the position $f(\tau(m))$ it 
   either moves further along $\Gamma$, if $l>m$, 
   or goes back along $\Gamma$, if $l < m$, until 
   it reaches the position $\tau(l) = z$; 
   while moving it writes $a,a^{-1},b$ 
   and $b^{-1}$ depending
   whether it moves right, left, up and down, 
   respectively. 
   For illustration let us consider 
   the group element 
   $h \in \mathbb{Z}_2 \wr \mathbb{Z}^2$ 
   shown in Fig.~\ref{z2wrz2fig1}: 
   a black square means that the lamp at 
   the current position is lit, i.e., $\varphi((x,y))=c$, 
   a white square means that the lamp at the 
   current position is unlit, i.e., 
   $\varphi((x,y))=e$, and a black circle shows
   the position of the lamplighter $z$ and 
   that the lamp at this position is lit, i.e., 
   $\varphi(z)=c$.  
   For this group element $h$ the 
   word $u$ is as follows: 
   \begin{equation*}
   \begin{split}
   u = a c b a^{-1} a^{-1} b^{-1}
   c b^{-1} c a a a b b b a^{-1} c a^{-1}
   a^{-1} a^{-1} b^{-1} b^{-1} c b^{-1} b^{-1}
   a a c a a a  b c b b c b c b c \\ 
   a^{-1} c 
   a^{-1} a^{-1} a^{-1} a^{-1}
   a^{-1} c b^{-1} c b^{-1} b^{-1} 
   b^{-1} b^{-1} c b^{-1} a c 
   a c a a a a a c,
   \end{split}
   \end{equation*} 
   and the word $v$ is as follows: 
   \begin{equation*}
   v = a^{-1} a^{-1} a^{-1}	a^{-1} a^{-1} a^{-1} a^{-1}
       b b b b b b a a a a a a b^{-1} b^{-1} b^{-1}
       b^{-1} b^{-1} a^{-1} a^{-1} a^{-1},
   \end{equation*}	 
   and the normal form of $h$ is a concatenation of $u$ 
   and $v$.     
   The described normal form of 
   $\mathbb{Z}_2 \wr \mathbb{Z}^2$ is 
   prefix--closed, so it is quasiregular.  
   Let us show that it satisfies the 
   $\sqrt{n}$--fellow traveler property. 

   We start with a pair of group elements 
   $g = (\varphi,z)$ and $ga = (\varphi, z')$, where 
   $z' = z + (1,0)$. 
   Let $l'$ be the integer for which 
   $\tau(l')=z'$. 
   Let $w_a = u_a v_a$ be the 
   normal form of $ga$.
   We first notice that  
   $d_A (\pi (w(n)),\pi (w_a(n)))$
   can be bounded from above by $|l - l'|$
   for every $n$: 
   \begin{equation} 
   \label{dAupperZ2} 
    d_A (\pi (w(n)),\pi (w_a(n))) 
    \leqslant |l - l'|.
   \end{equation}   
   For a given $p =(x,y) \in \mathbb{Z}^2$, let
   $r (p) = \max \{|x|,|y|\}$ and $k(p)$ be an integer 
   for which $\tau (k(p)) = p$. 
   One can notice the following lower and upper bounds
   for $k(p)$: 
   \begin{equation}
   \label{perim_ineqZ2} 
    (2r(p) -1)^2 - 1 \leqslant k(p) 
    \leqslant (2r(p)+1)^2 - 1. 
   \end{equation} 
   In particular, \eqref{perim_ineqZ2} implies that 
   $(2r(z)-1)^2 -1 \leqslant l \leqslant (2 r(z) + 1)^2 -1$ and $(2 r(z') -1)^2 -1 \leqslant l'           \leqslant (2 r(z') + 1)^2  - 1$. 
   Therefore,    
   $|l - l'| \leqslant  
    4(r(z')+ r(z)) |r(z')-r(z)+1|$.  
   Since $|r(z')-r(z)| \leqslant 1$, we obtain 
   that $|l - l'| \leqslant 8 (r(z) + r(z'))$. 
   Therefore,
   by \eqref{dAupperZ2}
   we obtain that for every $n$: 
   \begin{equation}
   \label{dAupperZ2_2}   
      d_A (\pi (w(n)),\pi (w_a(n))) \leqslant 
       16 r(z) + 8,\,\,
       d_A (\pi(w(n)),\pi(w_a(n))) \leqslant
       16 r(z') +8. 
   \end{equation}	
   Now we notice that $u = u_a$. 
   Therefore, if $n \leqslant  |u|$, 
   then $d_A (\pi (w(n)),\pi (w_a(n)))=0$. 
   For $n > |u|$ we consider the following three cases: 
   \begin{itemize} 
   	\item{Suppose $l'\geqslant l \geqslant m$. 
   	      If $|u| < n \leqslant |u| + (l - m)$, 
   	      then 
   	      $d_A (\pi (w(n)),\pi (w_a(n))) =0$.  
   	      By \eqref{dAupperZ2_2} we have that
   	      $d_A (\pi(w(n)),\pi(w_a(n))) \leqslant 16 r(z) + 8$. Therefore, 
   	      by \eqref{perim_ineqZ2}, 
   	      $d_A (\pi(w(n)),\pi (w_a(n))) \leqslant 8 (\sqrt{l+1} + 1) + 8$. 
   	      If $n > |u| + (l - m)$, then  $n \geqslant l$. 
   	      Therefore, 
   	      $d_A (\pi (w(n)),\pi (w_a(n))) \leqslant 8 \sqrt{n +1} + 16$.
   	    Similarly, suppose $l\geqslant l' \geqslant m$. 
   	  	If $|u| < n \leqslant |u| + (l' - m)$, 
   	  	then 
   	  	$d_A (\pi(w(n)),\pi (w_a(n))) =0$. 
   	  	By \eqref{dAupperZ2_2} we have that
   	  	$d_A (\pi(w(n)),\pi(w_a(n))) \leqslant 16 r(z') + 8$. Therefore, 
   	  	by \eqref{perim_ineqZ2}, 
   	  	$d_A (\pi (w(n)),\pi(w_a(n))) \leqslant 8 (\sqrt{l'+1} + 1) + 8$. 
   	  	If $n > |u| + (l' - m)$, then $n \geqslant l'$. Therefore,   
   	  	$d_A (\pi(w(n)),\pi(w_a(n))) \leqslant 8 \sqrt{n +1} + 16$.}
   	  \item{Suppose 
   	  	    $l' \geqslant m \geqslant l$. 
   	  	    By \eqref{dAupperZ2_2} we have 
   	  	    that $d_A(\pi (w(n)), 
   	  	    \pi(w_a(n))) 
   	  	    \leqslant 16 r(z) +8$. Therefore, by \eqref{perim_ineqZ2}, 
   	  	    $d_A (\pi(w(n)),\pi(w_a(n))) \leqslant 8 (\sqrt{l+1} + 1) + 8$. 
   	  	    If $n> |u|$, then  $n \geqslant l$. 
   	  	    Therefore,  $d_A (\pi(w(n)),\pi(w_a(n))) \leqslant 8 \sqrt{n +1} + 16$. 
   	   Similarly, suppose $l \geqslant m \geqslant l'$. 
   	   By \eqref{dAupperZ2_2} we have that 
   	   $d_A(\pi(w(n)), 
   	   \pi(w_a(n))) \leqslant 
   	   16 r'(z) +8$. Therefore, 
   	   by \eqref{perim_ineqZ2}, 
   	   $d_A (\pi(w(n)),\pi(w_a(n))) \leqslant 8 (\sqrt{l'+1} + 1) + 8$. 
   	   If $n>|u|$, then $n \geqslant l'$. Therefore,   
   	   $d_A (\pi(w(n)),\pi (w_a(n))) \leqslant 8 \sqrt{n +1} + 16$.}	     	    	     
      \item{Suppose $m \geqslant l' \geqslant l$. 
            If $|u|<n \leqslant |u| +  (m - l') $, 
            then
            $d_A (\pi (w(n)),\pi (w_a(n)))=0$.
            By \eqref{dAupperZ2_2} we have that 
            $d_A(\pi (w(n)), 
            \pi(w_a(n))) \leqslant 
            16 r'(z) +8$.
            Therefore, by \eqref{perim_ineqZ2}, 
            $d_A (\pi(w(n)),\pi(w_a(n))) \leqslant 8 (\sqrt{l'+1} + 1) + 8$. 
            If $n > |u| + (m - l')$, then 
            $n \geqslant l'$. Therefore, 
            $d_A (\pi(w(n)),\pi (w_a(n))) \leqslant 8 \sqrt{n +1} + 16$.
       Similarly, suppose $m \geqslant l \geqslant l'$. 
             If $|u| <  n \leqslant 
             |u| + (m - l)$, then 
             $d_A (\pi(w(n)),\pi(w_a(n)))=0$. 
             By \eqref{dAupperZ2_2} we have that 
             $d_A(\pi(w(n)), 
             \pi (w_a(n))) \leqslant 
             16 r(z) +8$. Therefore, 
             by \eqref{perim_ineqZ2}, 
             $d_A (\pi(w(n)),\pi(w_a(n))) \leqslant 8 (\sqrt{l+1} + 1) + 8$. 
             If $n > |u| + (m - l)$, $n \geqslant l$. 
             Therefore, $d_A (\pi (w(n)),\pi(w_a(n))) \leqslant 8 \sqrt{n +1} + 16$.}       
   \end{itemize}	 
   From these three cases we can see that $d_A (\pi(w(n)),\pi(w_a(n))) \preceq \sqrt{n}$. 
 
   Now let us consider a pair of group elements 
   $g = (\varphi,z)$ and $g c = 
   (\varphi',z)$, where 
   $\varphi(\gamma)= \varphi'(\gamma)$ if $\gamma \neq z$ 
   and $\varphi'(\gamma)= \varphi(\gamma) c$ if $\gamma = z$.  
   There are two different cases to consider: 
   $l \geqslant m$ and $l < m$. 
   Let $w_c = u_c v_c$ be the 
   normal form of $gc$.
   The case 
   $l \geqslant m$  is straightforward: 
   $d_A(\pi(w(n)), 
   \pi(w_c(n))) \leqslant 1$ for all $n$.

   Now let us consider the case $l<m$.  
   We denote by  $\widetilde{u}$ the following prefix 
   of $u$: 
   $\widetilde{u} = u_0 u_1 \dots u_{m-1} \alpha_m$. 
   If $n \leqslant |\widetilde{u}|$, then 
   $d_A (\pi(w(n)),\pi(w_c(n))) =0$.
   Suppose now that $n> |\widetilde{u}|$. 
   We assume that $\beta_{m} = \varepsilon$.
   Let $\widetilde{v}$ be the suffix of 
   $w(n)$ which follows 
   $\widetilde{u}$: 
   $w(n) = \widetilde{u}
                    \widetilde{v}$. 
   
   Let $z_0 = \tau(m)$ and  
   $z_1  = \tau(k)$ be a position of 
   the lamplighter for a group element 
   $\pi(w(n))$.
   We denote by $(x_0,y_0)$ and $(x_1,y_1)$  the coordinates of $z_0$ and $z_1$, respectively.
   There are two different cases: 
   $\widetilde{v} = 
    \alpha_{m+1} \beta_{m+1} \dots
    \beta_{k-1} \alpha_k$ or
   $\widetilde{v} = 
   \alpha_{m+1} \beta_{m+1} \dots \alpha_k \beta_k$, where $\beta_k = c$. Let us
   analyze these two cases: 
   \begin{itemize} 
    \item{Suppose $\widetilde{v} = 
    	\alpha_{m+1} \beta_{m+1} \dots \alpha_k$. Then 
    	$w_c(n) = 
    	\widetilde{u} c 
    	\alpha_{m+1} \beta_{m+1} \dots
    	\alpha_{k-1} \beta_{k-1}$, so   
    	we have that: 
    	\begin{equation*}
    	\pi(w_c(n)) = 
    	\pi(w(n))
    	a^{x_0-x_1}b^{y_0-y_1} c 
    	a^{x_1 - x_0}b^{y_1 - y_0} \alpha_k^{-1}.
    	\end{equation*}    
    Therefore, 
    $d_A (\pi(w(n)),\pi(w_c(n))) 
     \leqslant 2 |x_1 - x_0| + 2 |y_1 - y_0| + 2 
     \leqslant 2 (|x_1|  + |y_1| + |x_0|+ |y_0| + 1) 
     \leqslant 2 (2r(z_1) + 2 r(z_0) + 1)$. 
     By \eqref{perim_ineqZ2}, 
     $2 r(z_0) \leqslant \sqrt{m+1} + 1$ and 
     $2 r(z_1) \leqslant \sqrt{k+1} + 1$. 
     Therefore, 
     $d_A (\pi(w(n)),\pi(w_c(n))) 
      \leqslant 2 (\sqrt{k+1} + 1 + \sqrt{m+1} + 1 +1)    $.
     Since $n \geqslant k \geqslant m$, 
     we have that
     $d_A (\pi(w(n)),\pi(w_c(n))) 
      \leqslant 4 \sqrt{n+1} + 6$.}	 
    \item{Suppose $\widetilde{v} = 
    	  \alpha_{m+1} \beta_{m+1} \dots \alpha_k c$. Then
    	  $w_c(n) = 
    	  \widetilde{u} c 
    	  \alpha_{m+1} \beta_{m+1}\dots
    	  \beta_{k-1} \alpha_{k}$, so   
    	  we have that: 
    	  \begin{equation*}
    	  	 \pi(w_c(n)) = \pi (w(n))
    	  	 a^{x_0 - x_1} b^{y_0 - y_1}  c 
    	  	 a^{x_1 - x_0} b^{y_1- y_0} c. 	  	 
    	  \end{equation*}	
    	  Therefore, $d_A (\pi(w(n)),\pi(w_c(n))) 
    	  \leqslant 2 |x_1 - x_0| + 2 |y_1 - y_0| + 2$. 
    	  By exactly the same argument as in the 
    	  previous case we have 
    	  that  
    	  $d_A (\pi(w(n)),\pi(w_c(n))) 
    	  \leqslant 4 \sqrt{n+1} + 6$.}   
   \end{itemize}	
   Therefore, if $\beta_m = \varepsilon$, we 
   have that  
   $d_A (\pi(w(n)),\pi(w_c(n))) 
    \preceq \sqrt{n}$.    
   If $\beta_m = c$, then swapping the role 
   of $w(n)$ and $w_c(n)$ 
   in the argument above yields $d_A (\pi(w(n)),\pi(w_c(n))) 
   \preceq \sqrt{n}$. 
   Thus we finally proved that 
   $s(n) \preceq \sqrt{n}$.  
 \end{proof}

 \begin{remark} 	
    We note that for the normal form 
    in the proof of Theorem \ref{Z2wrZ2_quasireg_norm_form}
    the upper bound 
    $s(n) \preceq \sqrt{n}$ is sharp. 
    A proof of this is as follows. 
    For a given $m \geqslant 0$ let  
    $\varphi_m : \mathbb{Z}^2 \rightarrow \mathbb{Z}_2$ 
    be a function for which  
    $\varphi_m(\tau(m)) = c$ and 
    $\varphi_m(p) = e$ for
    $p \neq \tau(m)$.  
    We denote by 
    $g_m$ the group element 
    $g_m = (\varphi_m, (0,0))$. 
    Let $w_m$ and 
    $w_{ma}$ be the normal forms of  
    $g_m$ and $g_m a$, respectively. 
    Let $n = m+1$ and 
    $z_m = \tau(m)$. 
    Then $\pi(w_m(n))= 
    (\varphi_m, z_m)$ 
    and 
    $\pi(w_{ma}(n)) = (\varphi_0 , z_m)$. 
    Let $z_m = (x_m, y_m)$. 
    The distance $d_A (\pi(w_m(n)),
                       \pi (w_{ma}(n)))$  
                is equal to       
                $2(|x_m| + |y_m| +1)$.  
    Indeed, 
    in order to obtain $g_m a$ from $g_m$ 
    the lamplighter moves from the position 
    $(x_m,y_m)$ to the position $(0,0)$ choosing
    a shortest route, switch 
    on a lamp, moves back to the 
    position $(x_m,y_m)$ and switch off a lamp.
    In particular, $d_A (\pi(w_m(n)),
    \pi(w_{ma}(n))) \geqslant 2 r(z_m) +1$. 
    By \eqref{perim_ineqZ2}, 
    we have that $(2r(z_m) +1)^2 \geqslant m+1$. 
    Therefore, 
    $d_A (\pi(w_m(n)),
    \pi (w_{ma}(n))) \geqslant \sqrt{n}$. 
    This implies that $\sqrt{n} \preceq s(n)$.

 \end{remark}


\section{Discussion and Open Questions}
\label{conclusion_sec}

Theorems \ref{ssp_no_quasidesic_thm} and 
\ref{nonfp_no_quasigeodesic_thm}  
show that for a finitely presented 
group with the strongly--super--polynomial 
Dehn function or a non--finitely 
presented group there exists no 
quasigeodesic normal form satisfying the 
$f(n)$--fellow traveler property. 
The following question is apparent 
from these results.
\begin{enumerate}
\item{Is there a quasigeodesic normal 
form satisfying the $f(n)$--fellow traveler
property for some finitely presented 
non--automatic
group 
with the Dehn function which is not strongly--super--polynomial? Some interesting 
candidates to consider this question 
include, for example, the Heisenberg group 
$\mathcal{H}_3 (\mathbb{Z})$ and 
the higher Heisenberg groups 
$\mathcal{H}_{2k+1}(\mathbb{Z})$, 
$k > 1$.} 
\end{enumerate} 

\noindent Theorems \ref{BS_qusireg_normal_form}  and 
\ref{Z2wrZ2_quasireg_norm_form} show 
the existence of a quasiregular 
normal form satisfying the 
$f(n)$--fellow traveler property
for Baumslag--Solitar 
groups $BS(p,q)$, 
$1 \leqslant p < q$, and the wreath 
product $\mathbb{Z}_2 \wr \mathbb{Z}^2$.
We leave the following question for future 
consideration. 
\begin{enumerate}
\item[2.]{Is there a quasiregular normal 
form satisfying the $f(n)$--fellow traveler 
property for the fundamental group 
of a torus bundle over a circle 
$\mathbb{Z}^2 \rtimes_A \mathbb{Z}$, 
where $A \in \mathrm{GL}(2,\mathbb{Z})$ has 
two real eigenvalues not equal to $\pm 1$?  
Recall that the latter guarantees that
$\mathbb{Z}^2 \rtimes_A \mathbb{Z}$ has at least exponential Dehn function, so no 
quasigeodesic normal form satisfying the 
$f(n)$--fellow traveler property exists in 
this case.}
\end{enumerate} 

\noindent In addition to that there are other questions that might be worth considering. Is there a quasiregular normal form 
satisfying the $f(n)$--fellow traveler 
proper for $BS(p,q)$, 
$1\leqslant p < q$, for some 
$f \ll \log (n)$? Is there a quasiregular normal form 
satisfying the $f(n)$--fellow traveler 
proper for $\mathbb{Z}_2 \wr \mathbb{Z}^2$, 
$1\leqslant p < q$, for some 
$f \ll \sqrt{n}$? What are the other 
examples of groups which admit 
quasiregular normal forms satisfying the 
$f(n)$--fellow traveler property?     

\section*{Acknowledgments} 

The authors thank the anonymous reviewer for useful comments.

\bibliographystyle{splncs03}

\bibliography{weakly_fellow}

\end{document}